\documentclass[12pt]{amsart}

\usepackage[dvipsnames]{xcolor}
\usepackage{graphicx}
\usepackage{dsfont}
\usepackage{amsmath}
\usepackage{color}
\usepackage{caption}
\usepackage{subcaption}
\usepackage[T1]{fontenc}
\usepackage{mathtools}
\usepackage{amsfonts,amssymb,amscd}

\calclayout

\newtheorem{thm}{Theorem}[section]
\newtheorem{prop}[thm]{Proposition}
\newtheorem{lem}[thm]{Lemma}
\newtheorem{cor}[thm]{Corollary}

\newtheorem*{question}{Question}

\theoremstyle{definition}
\newtheorem{definition}[thm]{Definition}

\makeatletter
\renewcommand*\env@matrix[1][\arraystretch]{%
  \edef\arraystretch{#1}%
  \hskip -\arraycolsep
  \let\@ifnextchar\new@ifnextchar
  \array{*\c@MaxMatrixCols c}}
\makeatother

\theoremstyle{remark}

\newtheorem{rem}[thm]{Remark}

\theoremstyle{remark}

\theoremstyle{question}

\numberwithin{equation}{section}

\DeclareMathOperator{\dist}{dist} 

\begin{document}


\title{Oscillating Wandering Domains for Functions with Escaping Singular Values}


\author{Kirill Lazebnik}

\maketitle

\begin{abstract}
We construct a transcendental entire $f:\mathbb{C}\rightarrow\mathbb{C}$ such that (1) $f$ has bounded singular set, (2) $f$ has a wandering domain, and (3) each singular value of $f$ escapes to infinity under iteration by $f$.
\end{abstract}

\tableofcontents


\section{Introduction}
\label{introduction}

Associated with any transcendental entire function $f:\mathbb{C}\rightarrow\mathbb{C}$ is a dynamical  partition of $\mathbb{C}$ into two sets: the \emph{Fatou set} $\mathcal{F}(f)$ and its complement, the \emph{Julia set} $\mathcal{J}(f)$. The Fatou set $\mathcal{F}(f)$ is defined as the maximal region of normality for the family of iterates $(f^{n})_{n=1}^{\infty}$, and is itself further partitioned into connected open components termed \emph{Fatou components}. A Fatou component $U$ is said to be \emph{periodic} if $f^n(U)\subseteq U$ for some $n\geq1$, and \emph{pre-periodic} if $f^k(U)$ is periodic for some $k\geq0$. There are two immediate questions which arise in the study of $\mathcal{F}(f)$: (1) classifying, up to conjugacy, the dynamics of $f$ on any periodic Fatou component, and (2) determining whether all Fatou components of $f$ are pre-periodic. We first discuss (1). 

The classification of the dynamics of $f$ on periodic components of $\mathcal{F}(f)$ was given already by Fatou in \cite{Fat20}. It is remarkable that for each possible periodic component in this classification, there is a necessary (and, in most cases, easy to state) relationship with a \emph{singular value} of $f$: some point in the plane at which it is not possible to define all branches of $f^{-1}$. We denote the collection of singular values of $f$ by $S(f)$. The simplest example of the aformentioned relationship is that a \emph{basin of attraction} (a Fatou component on which $f$ is conjugate to dilation on $\mathbb{D}$ by a complex factor with modulus strictly smaller than one) must contain a singular value of $f$ (see, for instance, Theorem 37 in \cite{MR1626031}). 

The question (2) was answered for transcendental $f$ with finitely many singular values in \cite{EL92}, \cite{GK86} (using techniques of \cite{Sul85}), where it was shown that all Fatou components must be pre-periodic for such $f$. It was already known that non pre-periodic Fatou components, termed \emph{wandering domains}, could exist for $f$ with infinitely many singular values \cite{Bak76} (see also \cite{Her84}, \cite{EL87}, \cite{FH09}, \cite{Bis15}, \cite{MR3339086}, \cite{MR3579902}, \cite{2017arXiv171110629F}, \cite{2018arXiv180711820M}).  In analogy with the problem (1), an active line of research is in determining relationships between a wandering domain of a function $f$ and the singular values of $f$ (see, for instance, \cite{BFJK17}). One precise question in this area is as follows, where we note that $\mathcal{B}$ denotes the class of transcendental entire $f$ with \emph{bounded} singular set:

\begin{question}{\cite{MR3124942}}
\label{main_question}  Let $f\in\mathcal{B}$, and suppose that the singular values of $f$ tend to infinity uniformly under iteration, that is, $\lim_{n\rightarrow\infty} \inf_{s\in S(f)} |f^n(s)| = \infty$. Can $f$ have a wandering domain?

\end{question}

\noindent The present work is concerned with the following variant of Question \ref{main_question}:

\begin{question}{}
\label{variant}  Let $f\in\mathcal{B}$, and suppose that the singular values of $f$ tend to infinity under iteration. Can $f$ have a wandering domain?

\end{question}

Question \ref{main_question} was posed in a work in which the authors demonstrated the non-existence of wandering domains for a certain subclass of $\mathcal{B}$. Outside of this subclass, the existence of $f\in\mathcal{B}$ with a wandering domain was proven in \cite{Bis15}, and this is the approach that the present work most closely follows (for a different approach, see \cite{2018arXiv180711820M}). The wandering domain for the function of \cite{Bis15} is oscillating (see \cite{MR3525384}) and contains infinitely many singular values in its grand orbit.  Nevertheless, it was shown in \cite{2017arXiv171110629F} that, with appropriate modifications, a similar approach yields a \emph{univalent} wandering domain for a function $f\in\mathcal{B}$. In particular, there is a wandering Fatou component for the function $f$ of  \cite{2017arXiv171110629F} whose forward orbit contains no singular values of $f$. The constructions in \cite{Bis15}, \cite{MR3339086}, \cite{MR3579902}, \cite{2017arXiv171110629F}, \cite{2018arXiv180711820M} of wandering domains in class $\mathcal{B}$ do not answer Question \ref{variant} (nor Question \ref{main_question}): for $f$ as in \cite{Bis15}, \cite{MR3339086}, \cite{MR3579902}, \cite{2018arXiv180711820M} there are oscillating singular values, and the orbits of the singular values in \cite{2017arXiv171110629F} are not sufficiently well understood for this purpose. The present work shows that the answer to Question \ref{variant} is yes:


\begin{thm} \label{main_theorem} There exists an entire function $f \in \mathcal{B}$ with a wandering Fatou component, such that for all $s\in S(f)$, $f^n(s) \rightarrow \infty$ as $n\rightarrow\infty$. 

\end{thm}

\begin{rem} We note that for the function $f$ of Theorem \ref{main_theorem}, the convergence $f^n(s) \rightarrow \infty$ is not uniform in $s \in S(f)$. In other words, $\inf \{ \left| f^n(s) \right| \textrm{ } : \textrm{ } s \in S(f)\}  \not\rightarrow \infty$ as $n\rightarrow\infty$, and so Question \ref{main_question} remains open.

\end{rem}

The wandering Fatou component of $f$ as in Theorem \ref{main_theorem} is oscillating, whereas the singular values of $f$ escape. This is surprising in that it demonstrates that the dynamics of $f\in\mathcal{B}$ on a \emph{wandering} Fatou component can differ markedly from the dynamics of $f$ on $S(f)$. This is in contrast to the situation, already discussed, of the necessary relationship between the dynamics of $f$ on any \emph{periodic} Fatou component and the singular values of $f$. For a related result outside of class $\mathcal{B}$, we refer to \cite{MR2458806}.

\vspace{10mm}

We conclude the Introduction with a brief, non-technical outline of the proof of Theorem \ref{main_theorem}. We remark that the following is purely expository and is also meant to motivate the necessity of some of the more technical aspects of the present work. Precise definitions and arguments will follow in Sections \ref{preliminaries}-\ref{fixpoint_section}. We will refer to Figure \ref{fig:rough_idea} throughout our discussion.

\begin{figure}
\centering
\scalebox{.3}{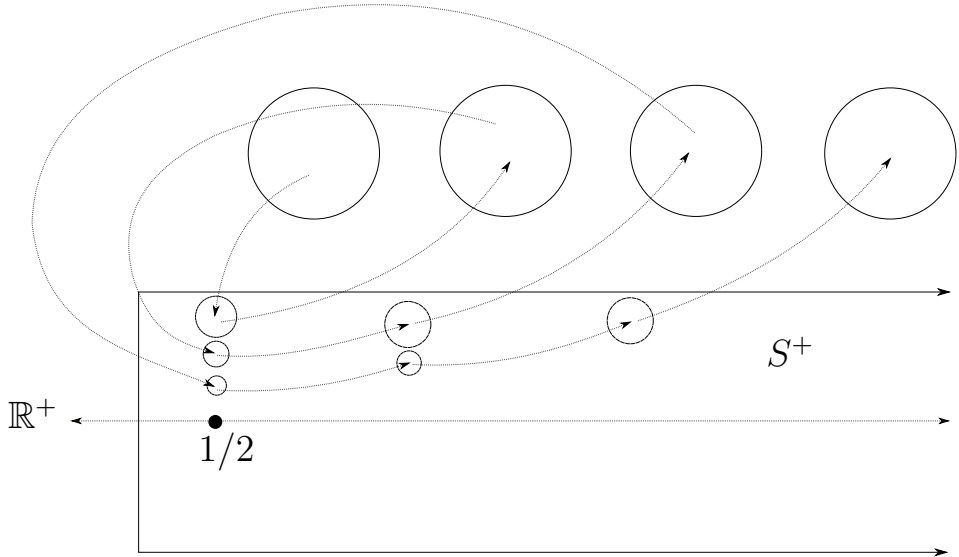}
\caption{Pictured is a schematic of the wandering domain for $f$ as in Theorem \ref{main_theorem}. The Euclidean discs pictured lying outside $S^+$ are subdiscs of the discs $(D_n)$, and are contained in a wandering domain for $f$. For the sake of clarity, this Figure is not to scale: the images of the subdiscs of $D_n$ are smaller and closer to $1/2$ than shown here.  }
\label{fig:rough_idea}       
\end{figure}

Let \[S^+ := \{ z\in\mathbb{C} : \textrm{Re}(z)>0\textrm{, } |\textrm{Im}(z)|<\pi/2 \}.\] We will define a quasiregular function $g:S^+\rightarrow\mathbb{C}$ (see (\ref{quasiregular_definition}) in Section \ref{model_map}) such that: \begin{itemize} \item[$\bullet$] $g(\mathbb{R}^+)\subset\mathbb{R}^+$, \item[$\bullet$] $g^n(x)\rightarrow\infty$ as $n\rightarrow\infty$ for $x\in\mathbb{R}^+$, and \item[$\bullet$] $\mathbb{C}\setminus[-1,1]\subset g(S^+)$. \end{itemize} We will consider a sequence of pairwise disjoint discs $D_n:=D(z_n,1)$ with radii $1$ and centers $z_n$ satisfying $\textrm{Im}(z_n)=\pi$ (see Section \ref{model_map}). The map $g$ is defined so that there is a subsequence of natural numbers $(\tilde{p}_n)_{n=1}^\infty$ (see Definition \ref{sequence_along_which_we} in Section \ref{model_map}) such that each $D_{\tilde{p}_n}$ has an $n^{\textrm{th}}$ preimage $g^{-n}(D_{\tilde{p}_n})$ close to $1/2$, where $g^{-1}$ denotes a branch of the inverse of $g|_{S^+}$. For the purposes of this outline and to simplify notation, we will assume $D_{\tilde{p}_n}=D_n$ (we will be more precise in Sections \ref{model_map}-\ref{fixpoint_section}).


We may extend the definition of the map $g|{S^+}$ to a quasiregular map $g: \cup_nD_n \rightarrow \mathbb{D}$ so that \begin{align}\label{desired_behavior} g\left(D\left(z_{n}, \frac{n}{n+1}\right)\right)\subset g^{-(n+1)}\left( D\left(z_{n+1}, \frac{n+1}{n+2}\right) \right) \textrm{ for all } n\geq1 \end{align} (see Figure \ref{fig:rough_idea}, and Sections \ref{disc_component_maps}, \ref{model_map}). This essentially describes the desired behavior for the wandering domain we wish to construct: a subdisc of $D_1$ is mapped to a preimage of a subdisc of $D_2$, whence $g(D_1)$ iterates to a subdisc of $D_2$. Then this subdisc of $D_2$ is mapped to a preimage of a subdisc of $D_3$, and so on. 

In order to construct an \emph{entire} function $f$ with a wandering domain, we will first extend the definition of $g$ (so far given only on $S^+$ and the discs $(D_n)_{n=1}^\infty$) to a quasiregular function defined on all of $\mathbb{C}$ via the Folding Theorem of Bishop \cite{Bis15}. The entire function of Theorem \ref{main_theorem} is then defined as $f:=g\circ\phi^{-1}$, where $\phi$ is a quasiconformal mapping obtained by applying the Measurable Riemann Mapping Theorem (see Theorem \ref{MRT}) to the Beltrami coefficient $g_{\overline{z}}/g_z$ of $g$ (see Section \ref{preliminaries}).

A fundamental difficulty is that the relation (\ref{desired_behavior}) need not hold when $g$ is replaced by $f$: the map $f$ differs from $g$ by precomposition with $\phi^{-1}$, and so even if $\phi^{-1}$ is ``close to the identity'', the maps $f^n$ and $g^n$ may still differ significantly for large $n$. We will prove that we can define $g$ such that the relation:

\begin{align}\label{desired_behavior2} f\circ\phi\left(D\left(z_n, \frac{n}{n+1}\right)\right)\subset f^{-(n+1)}\circ\phi\left( D\left(z_{n+1}, \frac{n+1}{n+2}\right) \right) \end{align}


\noindent holds in a subsequence of $n$ (see the proof of Theorem \ref{main_theorem} at the end of Section \ref{fixpoint_section}), whence it will follow that $\phi(D(z_1, \frac{1}{2}))$ is contained in a wandering Fatou component for $f$. It is not clear how to define $g$ such that (\ref{desired_behavior2}) holds. As already explained, the validity of relation (\ref{desired_behavior}) will not directly imply that a similar relation holds for $f$. Moreover, changing the definition of $g$ will change the correction map $\phi$ (via the Measurable Riemann mapping theorem), and hence $f$, in a non-explicit manner.

A procedure of defining $g$ so that the map $f:=g\circ\phi^{-1}$ has a wandering domain was described already in \cite{Bis15}, \cite{MR3339086}, \cite{MR3579902}, \cite{2017arXiv171110629F}. The strategy therein consists of carefully estimating the differences $|\phi(z)-z|$ for $z$ within the relevant region where the wandering domain of $f$ is being constructed, and ensuring that the differences $|\phi(z)-z|$ are sufficiently small so that a relation such as (\ref{desired_behavior}) for $g$ persists even for the function $f:=g\circ\phi^{-1}$. The contribution of the present work is in being able to also ensure that all critical values of $f$ iterate to $\infty$. This requires a more delicate analysis because one requires finer information on the behavior of $\phi$ rather than an estimate on $|\phi(z)-z|$: one needs fine control on the iterates $(f^n)_{n=1}^\infty$ at all critical values of $f$.  The approach presented here involves combining continuous dependence on parameters (Theorem \ref{parameters}) together with Brouwer's Fixpoint Theorem (Theorem \ref{fixpoint}). A similar approach was developed for a different purpose in \cite{2018arXiv180704581B}, and independently in \cite{2018arXiv180711820M}. We will comment further on our strategy in Remark \ref{purely_expository} (see Section \ref{fixpoint_section}) after we have established some necessary notation. 

In Section \ref{preliminaries} we will record some classical results we will make frequent use of. Section \ref{disc_component_maps} will focus on a definition for $g$ in the discs $(D_n)_{n=1}^\infty$. In Section \ref{model_map} we extend the definition of $g$ to all of $\mathbb{C}$. Section \ref{fixpoint_section} sets up an application of Brouwer's Fixpoint Theorem and contains the main technical contributions of the present work.






\vspace{2.5mm}

\noindent \emph{Acknowledgements.} The author would like to thank Chris Bishop, N\'uria Fagella, Xavier Jarque, and Lasse Rempe-Gillen for various conversations pertaining to the present work. 

\begin{rem} Our convention will be to use the notation $D(a,r):=\{z\in\mathbb{C} : |z-a| < r\}$ for the open Euclidean disk centered at $a\in\mathbb{C}$ of radius $r$, and $\overline{D}(a,r):=\{z\in\mathbb{C} : |z-a| \leq r\}$ for the closed Euclidean disk centered at $a\in\mathbb{C}$ of radius $r$.


\end{rem}

\section{Preliminaries}
\label{preliminaries}

In this Section we will list, for the reader's convenience, several classical results from function theory (Theorems \ref{thm:Koebe} and \ref{thm:Koebe_argument}) and from the theory of quasiconformal mappings (Theorems \ref{MRT}, \ref{parameters}, \ref{Lehto-Virtanen}), but we first start with the classical Brouwer Fixpoint Theorem:

\begin{thm}{\emph{\cite{MR1511582}}} \label{fixpoint}  Let $X\subset\mathbb{R}^n$ be non-empty, compact, and convex. Any continuous function $f : X \rightarrow X$ has a fixpoint.
\end{thm}

Theorems \ref{MRT} and \ref{parameters} below are (respectively) referred to as the Measurable Riemann Mapping Theorem, and continuous dependence on parameters. For proofs, history, and references, we refer the reader to Chapter 4 of \cite{MR2245223}. The last Theorem we will record that concerns quasiconformal mappings is Theorem \ref{Lehto-Virtanen}, and is exposited as Theorem 5.2 in \cite{MR0344463}. We first remark that for a $C^1$ function $g$, we have the definitions $g_z := (1/2)(g_x - i g_y)$ and $g_{\overline{z}} = (1/2)(g_x + i g_y)$. For functions $g$ which are quasiregular but not necessarily $C^1$, $g_z$ and $g_{\overline{z}}$ are defined using distributional derivatives (see, for instance, Chapter VI of \cite{MR0344463} for details).

\begin{thm}{}
\label{MRT}
If $\mu\in L^{\infty}(\mathbb{C})$ with $||\mu||_{L^\infty(\mathbb{C})}<1$, there exists a quasiconformal mapping $\phi:\mathbb{C}\rightarrow\mathbb{C}$ so that $\phi_{\overline{z}}/\phi_z=\mu$ a.e.. Moreover, given any other quasiconformal $\Phi:\mathbb{C}\rightarrow\mathbb{C}$ with $\phi_{\overline{z}}/\phi_z=\Phi_{\overline{z}}/\Phi_z$ a.e., there exists a conformal $\psi: \mathbb{C}\rightarrow\mathbb{C}$ so that $\Phi= \psi\circ\phi$.  

\end{thm}


\begin{thm}
\label{parameters}
Let $\mu \in L^{\infty}(\mathbb{C})$ and $(\mu_n)_{n=1}^\infty \in L^{\infty}(\mathbb{C})$ be such that there exists $0<k<1$ with $||\mu||_{L^\infty(\mathbb{C})}<k$ and $||\mu_n||_{L^\infty(\mathbb{C})}<k$ for all $n$. Denote by $\phi_\mu$ the unique quasiconformal solution of $\frac{\partial \phi_\mu}{\partial\bar{z}}=\mu\frac{\partial \phi_\mu}{\partial z}$ satisfying some fixed normalization, and similarly for $\phi_{\mu_n}$. If $\mu_n \rightarrow \mu$ a.e. as $n\rightarrow\infty$, then $\phi_{\mu_n} \rightarrow \phi_{\mu}$ as $n\rightarrow\infty$ uniformly on compact subsets. Consequently, for any fixed $z\in\mathbb{C}$, the map $L^\infty(\mathbb{C})\rightarrow\mathbb{C}$ given by $\mu\mapsto\phi_\mu(z)$ is continuous.
\end{thm}

\begin{thm}{\emph{\cite{MR0083025}}} \label{Lehto-Virtanen} Let $\phi_n:\mathbb{C}\rightarrow\mathbb{C}$ be a sequence of $K$-quasiconformal mappings converging to a quasiconformal mapping $\phi: \mathbb{C}\rightarrow\mathbb{C}$ with complex dilatation $\mu$ uniformly on compact subsets of $\mathbb{C}$. If the complex dilatations $\mu_n(z)$ of $\phi_n$ tend to a limit $\mu_\infty(z)$ almost everywhere, then $\mu_\infty(z)=\mu(z)$ almost everywhere.  

\end{thm}

The last two theorems we record in this Section are classical results from function theory. Theorem \ref{thm:Koebe} is a well known distortion estimate due to Koebe. Theorem \ref{thm:Koebe_argument} is due to Grunsky \cite{Gru32} and estimates the arguments of the quantities whose moduli are estimated in Theorem \ref{thm:Koebe}. We refer the reader to Sections II.4 and IV.1 of \cite{MR0247039} for proofs. 

\begin{thm}\label{thm:Koebe}
Let $F$ be a univalent function on the disk $D(a,r)$ for some $a\in\mathbb{C}$ and $r>0$.  Then
\begin{itemize}

\item[(a)] For all $z\in D(a,r)$,

\[
 \dfrac{r^2}{(r+|z-a|)^2}\leqslant \left|\dfrac{F(z)-F(a)}{F'(a)(z-a)}\right| \leqslant\dfrac{r^2}{(r-|z-a|)^2}.
 \]
 
 \item[(b)]  For all $z\in D(a,r)$,

\[
 \dfrac{1-\left| \frac{z-a}{r} \right|}{(1+|\frac{z-a}{r}|)^3}\leqslant \left| \frac{F'(z)}{F'(a)} \right| \leqslant\dfrac{1+\left| \frac{z-a}{r} \right|}{(1-|\frac{z-a}{r}|)^3}.
 \]
 
\end{itemize}

\end{thm}

\begin{thm}\label{thm:Koebe_argument} Let $F$ be a univalent function on the disk $D(a,r)$ for some $a\in\mathbb{C}$ and $r>0$.  Then

\item[(a)] For all $z\in D(a,r)$,

\[
\left| \arg\left(\dfrac{ F(z)-F(a) }{F'(a)(z-a)}\right) \right|\leqslant \log\left(\dfrac{r+|z-a|}{r-|z-a|}\right).
 \]
 
 \item[(b)]  For all $z\in D(a,r)$,

\[
\left| \arg\left(\dfrac{F'(z)}{F'(a)}\right) \right|\leqslant 2\log\left(\dfrac{r+|z-a|}{r-|z-a|}\right).
 \]

\end{thm}

We remark that the proof of Theorem \ref{main_theorem} will depend, as already described in the Introduction, on the \emph{Folding Theorem} of Bishop \cite{Bis15}. The crucial application of \cite{Bis15} will be in extending a quasiregular function $g$ defined on a subset of $\mathbb{C}$ to all of $\mathbb{C}$, such that: \begin{itemize} \item[$\bullet$] $||g_{\overline{z}}/g_z||_{L^\infty(\mathbb{C})}$ is independent of a number of parameters described in Section \ref{disc_component_maps}, and \item[$\bullet$] the singular values of $g$ are well-understood as described in Theorem \ref{g_extension}. \end{itemize} The techniques of Bishop were applied in a similar manner in the constructions of \cite{Bis15}, \cite{MR3339086}, \cite{MR3579902}, \cite{2017arXiv171110629F}, where a detailing of the techniques of \cite{Bis15} are given as they apply to the construction of entire functions with wandering domains. Thus we will forego a detailed discussion of the Folding Theorem, citing it only in the proof of Theorem \ref{g_extension}. 





\section{Disc-Component Maps }
\label{disc_component_maps}

In this Section we describe a quasiregular function of the unit disc (see Figure \ref{fig:D-component_map}), depending on several parameters, that we will use in constructing the function $f$ of Theorem \ref{main_theorem}. The term \emph{Disc-Component} comes from \cite{Bis15}, though we will not need to make explicit use of this definition here. We begin with a description of the map $\psi$ as given in \cite{2017arXiv171110629F}. The map $\psi$ will be an interpolation between $z\mapsto z^m$ on $|z|=1$ with $z\mapsto z^m+\delta z$ on $r\mathbb{D}$, where $r<1$ and $\delta>0$. In order to interpolate we will make use of a standard smooth bump function:

\[ b(x)=\begin{cases} 
      \exp(1+\frac{1}{x^2-1}) & \textrm{ if } 0\leq x < 1 \\
      0 & \textrm{ if } x\geq 1.
   \end{cases}
\]
We use the transformation $\phi(x):=\frac{x-r}{1-r}$ in order to define the modified smooth bump function: 
\[ \hat{\eta}(x)=\begin{cases} 
      1 & \textrm{ if }x\leq r \\
      b(\phi(x)) & \textrm{ if } r \leq x \leq 1 \\
      0 & \textrm{ if }x\geq 1,
   \end{cases}
\]
and we define $\eta(z):=\hat{\eta}\left(|z|\right)$.

\begin{lem} \label{interpolation} Let $\psi(z):= z^m + \delta z \eta(z)$ for $z\in\mathbb{D}$ with $r:= 1-(4\delta)/m$. There exist $m_0\in\mathbb{N}$, $\delta_0>0$, and  $k_0<1$ such that if $m>m_0$ and $\delta<\delta_0$, then $r>(\delta/m)^{1/(m-1)}$ and $||\frac{\psi_{\overline{z}}}{\psi_z}||_{L^\infty(\mathbb{D})}<k_0$. 

\end{lem}


\noindent For the proof of Lemma \ref{interpolation}, see Lemma 3.1 of \cite{2017arXiv171110629F}. We note here that the critical points of $\psi$ are $\left(\frac{-\delta}{m}\right)^{\left(\frac{1}{m-1}\right)}$, and the critical values of $\psi$ are $\delta\left(\frac{-\delta}{m}\right)^{\left(\frac{1}{m-1}\right)}\left(\frac{m-1}{m}\right)$. We will use the notation $\psi^{\delta,m}$ when we wish to emphasize the dependence of the map $\psi$ on the parameters $\delta, m$. In order to later apply Theorem \ref{fixpoint}, we will need the following Lemma:  

\begin{lem}\label{first_dependence_on_parameters}
Let $\psi(z):= z^m + \delta z \eta(z)$ as in Lemma \ref{interpolation}. The $L^\infty(\mathbb{D})$-valued map $\delta\mapsto\psi_{\overline{z}}/\psi_z$ is a continuous function of $\delta\in (0,\delta_0]$.
\end{lem}

\begin{proof} We first compute:

\begin{equation}\label{dilatation_expression} \psi_z(z)=mz^{m-1}+\delta\eta(z)+\delta z\eta_z(z)\textrm{, } \psi_{\overline{z}}(z)=\delta z\eta_{\overline{z}}(z). \end{equation}

\noindent Note that $\eta$ depends on a choice of $\delta$. Indeed, unraveling the definition, we have:


\[ \eta_\delta(z) = \mathds{1}_{[0,1-(4\delta)/m]}(|z|) + \mathds{1}_{[1-(4\delta)/m,1]}(|z|)\cdot\exp\left( 1 + \frac{1}{\left(\frac{|z|-(1-(4\delta)/m)}{1-(1-(4\delta)/m)}\right)^2-1} \right), \] 

\noindent where we have used the notation $\eta_\delta$ to emphasize the dependence on $\delta$. Given $\delta\in(0,\delta_0]$ and $(\delta_n)_{n=1}^\infty \in (0,\delta_0]$ such that $\delta_n\rightarrow\delta$, we claim that $\eta_{\delta_n}(z) \rightarrow \eta_{\delta}(z)$ as $n\rightarrow\infty$ uniformly over $z\in\mathbb{D}$. First observe that $\eta_{\delta_n}(z) \rightarrow \eta_{\delta}(z)$ as $n\rightarrow\infty$ pointwise over $z\in\mathbb{D}$. Since, moreover, the functions $(\eta_{\delta_n})_{n=1}^\infty$ are equicontinuous (this follows from the mean value theorem together with the derivative bound $|\hat{\eta}_{\delta_n}'(x)|\leq e/(4\delta_n/m)\leq e/(4(\inf_n \delta_n)/m)$ for $x\geq0$ as observed in the proof of Lemma \ref{interpolation} in \cite{2017arXiv171110629F}), the convergence  $\eta_{\delta_n} \rightarrow \eta_{\delta}$ is uniform (see, for instance, Exercise 7.16 of \cite{MR0385023}). Similar considerations yield that $(\eta_{\delta_n})_z\rightarrow(\eta_{\delta})_z$ in $L^{\infty}(\mathbb{D})$, and $(\eta_{\delta_n})_{\overline{z}}\rightarrow(\eta_{\delta})_{\overline{z}}$ in $L^{\infty}(\mathbb{D})$. The result then follows from (\ref{dilatation_expression}).







\end{proof}

\begin{rem} The $L^\infty(\mathbb{D})$-valued map $\delta\mapsto\psi_{\overline{z}}/\psi_z$ is \emph{not} an analytic function of $\delta$, as the reader may verify. Thus \emph{analytic} dependence on parameters will not be employable in the proof of Theorem \ref{main_theorem}, but continuous dependence on parameters (Theorem \ref{parameters}) will suffice. 
\end{rem}

Next we define a quasiconformal map $\beta$ whose purpose it will be to perturb the critical values of the map $\psi$ defined above. Enumerate, counter-clockwise, the $m^{\textrm{th}}$ roots of $-1$ as $\xi_1, ..., \xi_m$, where we assume $m$ is odd and take $\xi_1=\exp(\pi i/m)$. We define, for $\varepsilon>0$, the following subset of $\mathbb{C}^m$: 

\begin{equation}\label{convex_set} E_{\varepsilon} := \left\{ (r_1, ..., r_m) \in \prod_{j=1}^m \exp \left( \overline{D}(0,\varepsilon) \right) : \frac{r_{j+1}\xi_{j+1}-r_j\xi_j}{\xi_{j+1}-\xi_{j}} \in \exp \left( \overline{D}(0,\varepsilon) \right) \textrm{, } 1 \leq j \leq m  \right\},  \end{equation}

\noindent where we understand that $\xi_{m+1}:=\xi_1$ and $r_{m+1}=r_1$. The set $E_{\varepsilon}$ also depends on $m$, but we suppress it from the notation since the value of $m$ will always be understood from the context.

We define, for $R>1$, $m\in\mathbb{N}$, $\varepsilon>0$ and $(r_j)_{j=1}^m \in E_{\varepsilon}$, a map $\beta_{R, m, \varepsilon, (r_j)_{j=1}^m}$ on a subset of $\mathbb{C}$:

\begin{equation}\label{definition_of_beta} \beta_{R, m, \varepsilon, (r_j)_{j=1}^m}(z)=\begin{cases} 
      z & \textrm{ if }|z|\geq R, \\
      z & \textrm{ if }|z|\leq R^{-1}, \\
      r_1z & \textrm{ if } z = \xi_1, \\      
      \vdots \\
      r_mz & \textrm{ if } z = \xi_m. \\ 
   \end{cases}
\end{equation}

\noindent We will usually use the notation $\beta$, with the implicit dependence on parameters $R, m, \varepsilon, (r_j)_{j=1}^m$ understood. We will sometimes abbreviate $(r_j)$ in place of $(r_j)_{j=1}^m$, when $m$ is clear from the context. Our goal is to extend $\beta$ to a quasiconformal map of the complex plane whose dilatation has an upper bound which is essentially independent of $m\in\mathbb{N}$, $R>1$ and $(r_j) \in E_{\varepsilon}$, provided $\varepsilon>0$ is sufficiently small depending only on $R$ and not on $m$. This is formulated precisely and proven in Proposition  \ref{beta_extension} below, but we will first need to record the following preliminary computation:

\begin{lem}\label{dilatation_calculation} Let $T_z, T_w$ be triangles with vertices $z_1, z_2, z_3$ and $w_1, w_2, w_3=z_3$, respectively, as shown in Figure \ref{fig:dilatation_calculation} with $z_1,z_2 \in i\mathbb{R}$ and $\emph{Im}(z_1)=\emph{Im}(z_3)$. The affine map $L(z)$ sending $z_1, z_2, z_3$ to $w_1, w_2, w_3$, respectively,  satisfies

\begin{equation}\label{dilatation_calculation_formula} \left\lVert \frac{L_{\overline{z}}}{L_z} \right\rVert_{L^\infty(\mathbb{C})}< \dfrac{\left|\frac{z_1-w_1}{z_3-z_1}\right| + \left| 1- \frac{w_2-w_1}{z_2-z_1} \right| }{2-\left|\frac{z_1-w_1}{z_3-z_1}\right|-\left| 1- \frac{w_2-w_1}{z_2-z_1} \right|} \end{equation}

\end{lem}

\begin{proof} It suffices to bound the dilatation of the affine map $z\mapsto az+b\overline{z}+c$ sending the translate of $T_z$ by $-z_1$ to the translate of $T_w$ by $-z_1$. The coefficients $a,b,c$ are given by

\begin{equation}\label{dilatation_expression_continuity} a=\frac{1}{2}\left[ \frac{z_3-w_1}{z_3-z_1} + \frac{w_2-w_1}{z_2-z_1} \right]\textrm{,   } b=\frac{1}{2}\left[ \frac{z_3-w_1}{z_3-z_1} - \frac{w_2-w_1}{z_2-z_1} \right]\textrm{,    } c=w_1-z_1, \end{equation}

\noindent whence the inequality (\ref{dilatation_calculation_formula}) follows from applying the triangle inequality to both numerator and denominator of $|b/a|$.

\end{proof}

\begin{figure}
\centering
\scalebox{.6}{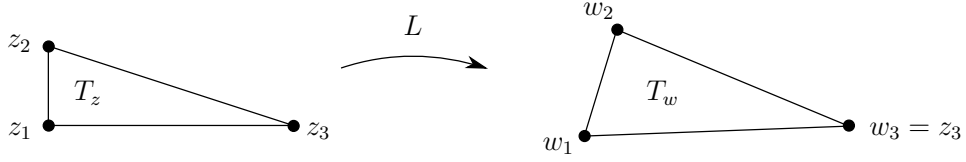}
\caption{ In order to calculate the dilatation of the map $\beta$ it will be convenient to have an expression for the dilatation of the affine map sending the triangle $T_z$ to $T_w$ as in Lemma \ref{dilatation_calculation}.  }
\label{fig:dilatation_calculation}       
\end{figure}

\begin{prop}\label{beta_extension} There exist constants $k_1<1$, $m_0'\in\mathbb{N}$ and $n_0\in\mathbb{N}$ such that if $1\leq R<3/2$, $m>m_0'$, and $(r_j)_{j=1}^m \in E_\varepsilon$ with $\varepsilon:=\log(\hspace{-1mm}\sqrt[\leftroot{-1}\uproot{2}n_0]{R})$, then the map $\beta=\beta_{R, m, \varepsilon, (r_j)_{j=1}^m}$ defined in (\ref{definition_of_beta}) may be extended to a quasiconformal map $\beta:\mathbb{C}\rightarrow\mathbb{C}$ such that 

\[ \left\lVert \frac{\beta_{\overline{z}}}{\beta_z} \right\rVert_{L^\infty(\mathbb{C})} < k_1. \]

\noindent Moreover, the $L^\infty(\mathbb{C})$-valued map $(r_j)_{j=1}^m\mapsto\beta_{\overline{z}}/\beta_z$ is continuous as a function of $(r_j)_{j=1}^m \in E_{\varepsilon}$. 


\end{prop}



\begin{proof} See Figure \ref{fig:beta_extension}: we define a $2\pi i$-periodic, piecewise-linear map $\hat\beta$ in the covering space $|\textrm{Re}(z)|<\log r$ of $R^{-1}<|z|<R$ such that $\hat\beta$ descends to an extension of the map $\beta$ with the desired properties. The definition is also illustrated in Figure  \ref{fig:beta_extension}. There are two triangulations of $|\textrm{Re}(z)|<\log R$ shown: the left-hand side is triangulated with vertices in $(\log(R^{-1} \xi_j) )_{j=1}^m$, $(\log(\xi_j) )_{j=1}^m$, $(\log(R \xi_j) )_{j=1}^m$ whereas the right-hand side is triangulated with vertices in $(\log(R^{-1} \xi_j) )_{j=1}^m$, $(\log(r_j\xi_j) )_{j=1}^m$, $(\log(R \xi_j) )_{j=1}^m$. The map $\hat\beta$ is defined piecewise: in each triangle on the left-hand side of Figure \ref{fig:beta_extension}, $\hat\beta$ is the affine map to the corresponding triangle on the right-hand side of Figure \ref{fig:beta_extension}.

\begin{figure}
\centering
\scalebox{.5}{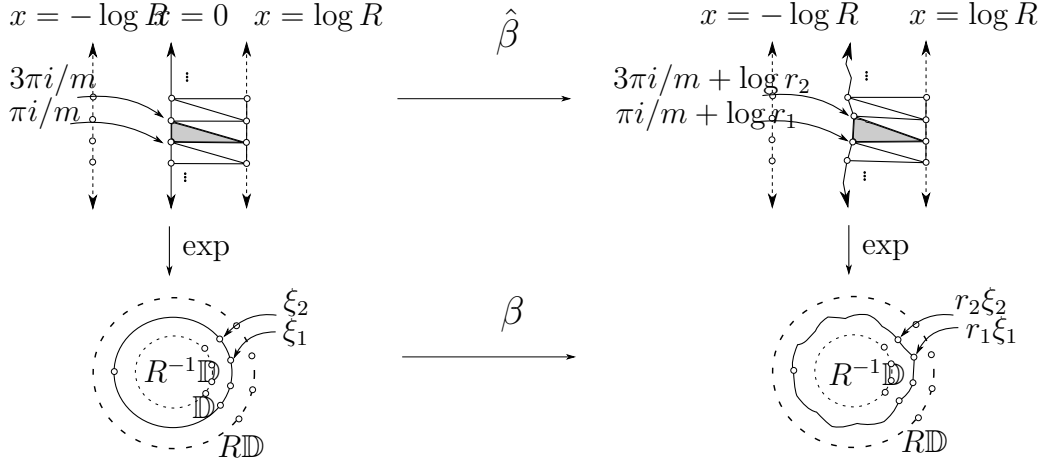}
\caption{Illustrated is the strategy in the proof of Proposition \ref{beta_extension}. The definition (\ref{definition_of_beta}) of $\beta$ is extended by defining a piecewise-linear map $\hat{\beta}$ in the covering space of $R^{-1}<|z|<R$.}
\label{fig:beta_extension}       
\end{figure}

It follows from the definition that $\hat\beta$ is $2\pi i$-periodic, $\hat\beta(\log\xi_j)=\hat\beta(\log(r_j\xi_j))$ for $1\leq j\leq m$, and $\hat\beta(z)\equiv z$ for $|\textrm{Re}(z)|=\log R$, so that $\hat\beta$ descends to a map which is an extension of $\beta$ in $R^{-1}<|z|<R$. It remains to verify the bound on the dilatation of this extension, for which it will suffice to bound the dilatation of the affine map between any two corresponding triangles pictured in Figure \ref{fig:beta_extension}. We will use Lemma \ref{dilatation_calculation} to perform the calculation for the triangles $T_1, T_2$ shaded in Figure \ref{fig:beta_extension} with vertices $z_1=\pi i/m$, $z_2=3\pi i/m$, $z_3=\log R + \pi i/m$ and $w_1=\pi i/m+\log r_1$, $w_2=3\pi i/m + \log r_2$, $w_3=z_3$. The calculation for the other triangles is similar. We have:

\begin{equation}\label{bound_1} \left| \frac{z_1-w_1}{z_3-z_1} \right| = \left| \frac{\log r_1}{\log R} \right| \leq \left| \frac{\log(\sqrt[\leftroot{-1}\uproot{2}n]{R}) }{\log R} \right| = \frac{1}{n}, \end{equation}


\noindent for $r_1\in\exp(\overline{D}(0,\log(\sqrt[\leftroot{-1}\uproot{2}n]{R}))$. Furthermore, 

\begin{eqnarray*}\label{bound_2}
\left| \frac{w_2-w_1}{z_2-z_1} \right|
& = & 
\left| \frac{\log(r_2\xi_2) - \log(r_1\xi_1)}{\log\xi_2-\log\xi_1} \right|
\\  & = &
\left| \frac{\log(r_2\xi_2) - \log(r_1\xi_1)}{\log'(r_1\xi_1)(r_2\xi_2-r_1\xi_1)}\right|\cdot\left|\frac{\log'(\xi_1)(\xi_2-\xi_1)}{\log(\xi_2)-\log(\xi_1)}\right|\cdot\left|\frac{r_2\xi_2-r_1\xi_1}{\xi_2-\xi_1}\right| \cdot \left|\frac{\log'(r_1\xi_1)}{\log'(\xi_1)}\right|
\\   & \leq & 
\frac{R^{-2/n}}{\left(R^{-1/n} - \left| \frac{r_2\xi_2-r_1\xi_1}{\xi_2-\xi_1} \right| \cdot \left|\xi_2-\xi_1\right|\right)^2}\cdot \left(1+\left|\xi_2-\xi_1\right|\right)^2\cdot R^{1/n} \cdot \frac{1+\left| r_1-1 \right|}{\left( 1- \left|r_1-1\right| \right)^3}
\\ & \leq & 
\frac{1}{\left(1 - R^{2/n} \cdot \left|\xi_2-\xi_1\right|\right)^2}\cdot \left(1+\left|\xi_2-\xi_1\right|\right)^2\cdot R^{1/n} \cdot \frac{R^{1/n}}{\left(2-R^{1/n} \right)^3},
\end{eqnarray*}

\noindent where the first inequality follows from Theorem \ref{thm:Koebe} and (\ref{convex_set}), and the second inequality follows from (\ref{convex_set}). Note that as $m\rightarrow\infty$, $\left|\xi_2-\xi_1\right|\rightarrow0$, and for $R<3/2$, $\sqrt[\leftroot{-1}\uproot{2}n]{R} < \sqrt[\leftroot{-1}\uproot{2}n]{3/2} \rightarrow 1$ as $n\rightarrow\infty$. Thus for any $s>1$, we have that $|(w_2-w_1)/(z_2-z_1)|<s$ for all sufficiently large $m$, $n$, and $\varepsilon:=\log(\sqrt[\leftroot{-1}\uproot{2}n]{R})$. Similarly, by using the left-hand sides of the inequalities in Theorem \ref{thm:Koebe}, we can show that $s^{-1}<|(w_2-w_1)/(z_2-z_1)|<s$ for all sufficiently large $m$, $n$, and $\varepsilon:=\log(\sqrt[\leftroot{-1}\uproot{2}n]{R})$. Lastly, by using the analogous estimates of Theorem \ref{thm:Koebe_argument} to estimate the argument of $(w_2-w_1)/(z_2-z_1)$, we can ensure that for any $\varepsilon'>0$, we have $\left|\arg((w_2-w_1)/(z_2-z_1))\right| < \varepsilon'$ for sufficiently large $m$, $n$, and $\varepsilon:=\log(\sqrt[\leftroot{-1}\uproot{2}n]{R})$. This means that we can fix $m_0'$ and $n_0$ so that for $m>m_0'$, $R<3/2$ and $\varepsilon:=\log(\hspace{-1mm}\sqrt[\leftroot{-1}\uproot{2}n_0]{R})$, we have that $|1-(w_2-w_1)/(z_2-z_1)|<1/10$ (the constant $1/10$ can be replaced here with any positive real number, perhaps by allowing for larger $m_0'$, $n_0$). Ensure furthermore that $n_0>10$ so that the right-hand side of (\ref{bound_1}) is less than $1/10$. Thus from (\ref{dilatation_calculation_formula}), we see that


\begin{equation}
\left\lVert \frac{\beta_{\overline{z}}}{\beta_z} \right\rVert_{L^\infty(\mathbb{C})} \leq \frac{1/10 + 1/10}{2-1/10-1/10} = \frac{1}{9} =: k_1. 
\end{equation}

\noindent The statement of continuity of the $L^\infty(\mathbb{C})$-valued map $(r_j)_{j=1}^m\mapsto\beta_{\overline{z}}/\beta_z$ follows from the expression (\ref{dilatation_expression_continuity}).

\end{proof}

It will be necessary to perturb the \emph{rescaled} $m^{\textrm{th}}$ roots of $-1$, so that for $\delta>0$ we make the definition $\beta^\delta(z):=\delta\beta(z/\delta)$. One has that $\left\lVert \beta_{\overline{z}}/\beta_z \right\rVert_{L^\infty(\mathbb{C})} = \left\lVert \beta^\delta_{\overline{z}}/\beta^\delta_z \right\rVert_{L^\infty(\mathbb{C})}$, and $\beta^\delta(\delta\xi_j)=r_j\delta\xi_j$ for $1\leq j\leq m$, as needed. In order to apply Theorem \ref{fixpoint}, we will need to establish the set $E_{\varepsilon}$ is convex:

\begin{lem} \label{convexity} For any $m\in\mathbb{N}$ and $0<\varepsilon<1$, $E_{\varepsilon}$ is a convex subset of $\mathbb{C}^m$.

\end{lem}

\begin{proof} Denote $h(z):=\exp(z)$. It is readily verified that  \[ \textrm{Re}\left( 1+ zh''(z)/h'(z) \right) > 0 \textrm{ for } z\in\mathbb{D},\] so that by Theorem 2.11 of \cite{Dur83},  $\exp(\overline{D}(0,1))$ is convex. The same Theorem 2.11 of \cite{Dur83} applied to a rescaled version of $h(z)=\exp(z)$ similarly shows that $\exp(\overline{D}(0,\varepsilon))$ is convex for any $0<\varepsilon<1$. We fix $m\in\mathbb{N}$ and $0<\varepsilon<1$ for the remainder of the proof.


Next we show that $t(r_1, ..., r_m)+(1-t)(r_1', ..., r_m') \in E_\varepsilon$ for any choice of $0\leq t\leq1$ and $(r_1, ..., r_m)$, $(r_1', ..., r_m') \in E_\varepsilon$. Note that $t(r_1, ..., r_m)+(1-t)(r_1', ..., r_m') \in \prod_{j=1}^m \exp(\overline{D}(0,\varepsilon))$ since $\prod_{j=1}^m \exp(\overline{D}(0,\varepsilon))$ is convex (it is a product of convex sets). That the other condition in (\ref{convex_set}) is satisfied follows from the calculation:


\[ \frac{\left[tr_{j+1}+(1-t)r_{j+1}'\right]\xi_{j+1}-\left[ tr_j+(1-t)r_j' \right]\xi_j}{\xi_{j+1}-\xi_j} = t\frac{r_{j+1}\xi_{j+1}-r_j\xi_j}{\xi_{j+1}-\xi_j} + (1-t)\frac{r_{j+1}'\xi_{j+1}-r_j'\xi_j}{\xi_{j+1}-\xi_j}. \]

\end{proof}

We will henceforth suppress the parameter $\varepsilon$ in the definition of $\beta=\beta^\delta_{R, m, \varepsilon, (r_j)}$, as we will always choose $\varepsilon:=\log( \hspace{-1mm}\sqrt[\leftroot{-1}\uproot{2}n_0]{R})<1$ as in Proposition \ref{beta_extension}.  Lastly, we recall, from \cite{Bis15}, the definition of a quasiconformal map $\rho_w:\mathbb{D}\rightarrow\mathbb{D}$ which is the identity on $|z|=1$, conformal on a region containing $0$, and perturbs the origin to $w$:

 \[ \rho_w(z)=\begin{cases} 
      z+w & \textrm{ if } 0\leq|z|\leq1/8 \\
      z\frac{(8|z|-1)}{7}+(z +w)\frac{8-8|z|}{7} & \textrm{ if } 1/8\leq|z|\leq 1.
   \end{cases}
\]

\begin{lem} \label{rho} There exists a constant $k_2<1$ independent of $w\in \overline{D}(0,3/4)$ such that $||\frac{(\rho_w)_{\overline{z}}}{(\rho_w)_z}||_{L^\infty(\mathbb{D})}<k_2$. The $L^\infty(\mathbb{D})$-valued map $w\mapsto (\rho_w)_{\overline{z}}/(\rho_w)_{z}$ is continuous as a function of $w\in \overline{D}(0,3/4)$.

\end{lem}

\noindent For the proof of Lemma \ref{rho}, see Lemma 3.4 of \cite{2018arXiv180711820M} or Section 3 of \cite{MR3339086}. Let \[ \Lambda(\delta, m):=\delta(\delta/m)^{1/(m-1)}(m-1)/m. \] We will consider, in ensuing sections, the following composition:

\begin{equation}\label{composition} \iota^{w, \delta, m, R, (r_j)} := \rho_w \circ \beta^{\Lambda(\delta, m)}_{R, m-1, (r_j)_{j=1}^{m-1}} \circ \psi_{\delta, m}: \mathbb{D} \rightarrow\mathbb{D}
\end{equation}

\noindent (see Figure \ref{fig:D-component_map}) where the terms $\Lambda(\delta, m)$ and $m-1$ in the second factor are chosen so that $\beta$ perturbs precisely the critical values of $\psi_{\delta, m}$. We will sometimes suppress the superscripts in (\ref{composition}) and simply write $\iota$. We have established in Lemma \ref{interpolation}, Proposition \ref{beta_extension}, and Lemma \ref{rho} a bound on the dilatation of (\ref{composition}) which is essentially independent of the parameters, and we will wish to vary those parameters so that the support of the dilatation of (\ref{composition}) is as small as desired:

\begin{prop}\label{small_support} Let $s<1$, $1/16>\tilde\delta>0$, and $1<\tilde{R}<1/s$. Then there exists $m_0''\in\mathbb{N}$ (depending on $s$, $\tilde\delta$, $\tilde{R}$) such that if $m>m_0''$, $1\leq R<\tilde{R}$, $(r_j)_{j=1}^{m-1} \in E_{\varepsilon}$, $w\in \overline{D}(0,3/4)$ and $1/16\geq\delta\geq\tilde\delta$, then $\emph{supp}\left( \iota^{w, \delta, m, R, (r_j)}_{\overline{z}}\right) \subset \{ z \in \mathbb{D} : |z|>s\}$. 

\end{prop}

\begin{proof} The map $\psi_{\delta, m}(z)$ is holomorphic, by definition, for $z\in\{z\in\mathbb{D}: |z|<1-(4\delta)/m \} \subset \{ z \in \mathbb{D} : |z| < 1-1/(4m) \}$, and $1-4/m > s$ for sufficiently large $m$. Next we consider the map $\rho_w(z)$, which is holomorphic for $|z| < 1/8$. Note that for sufficiently large $m$, $s^m<1/16$, whence $|z^m+\delta z| < s^m + \delta < 1/8$ for $|z|<s$ and hence the pullback of the dilatation of $\rho_w$ under $\beta \circ \psi_{\delta, m}$ is supported in $|z|>s$ (note that $\beta(z) \equiv z$ for $|z|>1/8$ and small $R$). 

Lastly we consider the pullback of $\textrm{supp}(\beta_{\overline{z}})=\{z\in\mathbb{D}: R^{-1}\Lambda(\delta, m) < |z| < R\Lambda(\delta, m)\}$ under $\psi_{\delta, m}$. We want to show that for large $m$, $\tilde{R}>R>1$ and $|z|<s$, we have $|z^m+\delta z| < R^{-1}\Lambda(\delta, m)$. Well since $|z^m+\delta z| < s^m+\delta s $, it suffices to show $s^m+\delta s < R^{-1}\Lambda(\delta, m)$, which can be rearranged to $s^m < \delta\left[ R^{-1}(\delta/m)^{1/(m-1)}(m-1)/m -s \right]$. We have

\begin{equation}\hspace{-5mm} \delta\left(R^{-1}(\delta/m)^{1/(m-1)}(m-1)/m - s\right)> \tilde\delta\left(\tilde{R}^{-1}(\tilde\delta/m)^{1/(m-1)}(m-1)/m - s\right) \xrightarrow{m\rightarrow\infty} \tilde\delta\left(\tilde{R}^{-1} -s\right)>0, \nonumber
\end{equation}

\noindent whereas $s^m\rightarrow0$ as $m\rightarrow\infty$. It follows that for sufficiently large $m$ and $1<R<\tilde{R}$, we have that the pullback of $\textrm{supp}(\beta_{\overline{z}})$ under $\psi_{\delta, m}$ is contained in $\{ z \in \mathbb{D} : |z|>s\}$.

\end{proof}

\begin{figure}
\centering
\hspace*{-0mm}\scalebox{.68}{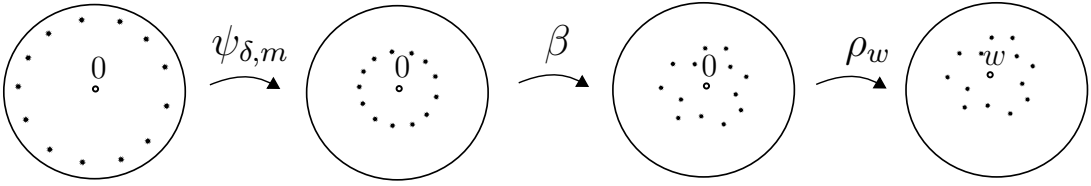}
\caption{ This figure illustrates the quasiregular function $\iota: \mathbb{D}\rightarrow\mathbb{D}$ as in (\ref{composition}). The $m-1$ critical points and their images are marked. Note that $\iota|_{\partial\mathbb{D}}(z)=z^m$. For the sake of clarity, the critical values are pictured at a larger scale than as considered in Proposition \ref{small_support}. }
\label{fig:D-component_map}       
\end{figure}

\begin{prop}\label{continuity_into_L^infty} For fixed $m\in \mathbb{N}$ and  $R>1$, the $L^{\infty}(\mathbb{D})$-valued map 
\begin{equation} (w,\delta,(r_j)) \mapsto  \left(\iota^{w, \delta, m, R, (r_j)}_{\overline{z}}\right) \big/ \left(\iota^{w, \delta, m, R, (r_j)}_z\right) \nonumber\end{equation} 

\noindent is continuous as a function of $(w,\delta,(r_j)_{j=1}^{m-1}) \in \overline{D}(0,3/4) \times (0,\delta_0] \times E_{\varepsilon}$ for $\delta_0$ as in Lemma \ref{interpolation}.
\end{prop}

\begin{proof} This follows from the continuity of the $L^{\infty}(\mathbb{D})$-valued maps of Lemma \ref{first_dependence_on_parameters}, Proposition \ref{beta_extension}, Lemma \ref{rho}, and the transformation formula: (see, for instance, Section IV.5.2 of \cite{MR0344463})

\[ \mu_{\phi\circ\chi}(z) = \frac{\mu_{\chi}(z) + \mu_{\phi}(\chi(z))e^{-2i\arg \chi_z(z)} }{1+\mu_{\chi}(z)\mu_{\phi}(\chi(z))e^{-2i\arg \chi_{\overline{z}}(z)}} \] 

\noindent for the dilatation $\mu_{\phi\circ\chi}$ of the composition of two quasiconformal maps $\phi, \chi: \mathbb{D}\rightarrow\mathbb{D}$. 
\end{proof}

\begin{rem} Recall that for $\delta>0$, we defined $\beta^{\delta}(z):=\delta\beta(z/\delta)$. We will have occasion to consider the degenerate case $\delta=0$, where we define $\beta^{0}(z):=z$. Note that $\psi_{0, m}(z)=z^m$ either by convention or suitable interpretation of the definition in Lemma \ref{interpolation}, so that for $\delta=0$ the mapping (\ref{composition}) becomes $z\mapsto \rho_w(z^m)$. 
\end{rem}

\section{A Base Family of Quasiregular Maps }
\label{model_map}

In this Section, we construct a family of quasiregular maps $g$ depending on several sets of parameters and provide relevant estimates. In the next Section, we prove that for some particular choice of these parameters, $g \circ \phi^{-1}$ is the desired function $f$ in the statement of Theorem \ref{main_theorem}, where $\phi^{-1}$ is an appropriately normalized straightening map of Theorem \ref{MRT}. This Section largely follows Section 4 of \cite{2017arXiv171110629F}, and we will omit those proofs which can be found there.

We define the horizontal half-strip $S^+ := \{ z\in\mathbb{C} : \textrm{Re}(z)>0, |\textrm{Im}(z)|<\pi/2 \}$, points $z_n :=  a_n +i\pi \sim n\pi +i\pi$ (see Section 4.1 of \cite{2017arXiv171110629F} for a precise definition of the points $z_n$), and discs $D_n := \overline{D}(z_n, 1)$. One defines the quasiregular map 

\begin{equation} \label{quasiregular_definition} g(z)=\begin{cases} 
      \sigma(\lambda\sinh(z)) & \textrm{ if } z\in S^+ \\
       \iota^{w_n, \delta_n, m_n, R_n, (r_j^n)}(z-z_n) & \textrm{ if } z \in D_n, \\
   \end{cases}
\end{equation}

\noindent where $\sigma(z)\equiv\exp(z)$ for $\textrm{Re}(z)>2\pi$ (see Section 4.1 of \cite{2017arXiv171110629F} for further discussion of the map $\sigma$). We have emphasized the dependence in the definition of $g$ on several sets of parameters: $\lambda, w, R, (r_j)_{j=1}^{m-1}, \delta, m$ (see Section \ref{disc_component_maps} for a discussion of the parameters $w, R, (r_j), \delta, m$). We have noted in (\ref{quasiregular_definition}), furthermore, that $w, R, (r_j), \delta, m$ are allowed to depend on $n$. We will use the notation $\boldsymbol{w}$ to denote the vector $(w_1, w_2, ...)$, and similarly for $\boldsymbol{\delta}$, $\boldsymbol{R}$, $\boldsymbol{m}$. We will use either of the notations $w_n$ or $\boldsymbol{w}(n)$ to denote the $n^{\textrm{th}}$ element of $\boldsymbol{w}$, and similarly for $\boldsymbol{\delta}$, $\boldsymbol{R}$, $\boldsymbol{m}$. We will use the notation $(\boldsymbol{r}_j)$ to denote the sequence $(( \boldsymbol{r}_j(k) )_{j=1}^{m_k-1})_{k=1}^{\infty}$ of vectors $(r_j(1))_{j=1}^{m_1-1} \in \mathbb{C}^{m_1-1}, (r_j(2))_{j=1}^{m_2-1} \in \mathbb{C}^{m_2-1}$, .... The following is an application of the Folding Theorem of \cite{Bis15}:


\begin{thm}\label{g_extension} There exist $m_0'''\in\mathbb{N}$, $\delta_0>0$, $n_0\in\mathbb{N}$ and  $k_3<1$ such that if \[ \boldsymbol{m}(k)>m_0'''\textrm{, } 0\leq\boldsymbol{\delta}(k)<\delta_0\textrm{, } 1\leq\boldsymbol{R}(k)<3/2\textrm{, } (\boldsymbol{r}_j(k))_{j=1}^{m_k-1} \in E_{\hspace{-1mm}\sqrt[\leftroot{-5}\uproot{2}n_0]{\boldsymbol{R}(k)}}\textrm{ and } \boldsymbol{w}(k)\in \overline{D}(0,3/4)\] for all $k\in\mathbb{N}$, then, for any $\lambda>1$, $g$ as in (\ref{quasiregular_definition}) may be extended to a quasiregular map $g:\mathbb{C}\rightarrow\mathbb{C}$ such that $||g_{\overline{z}}/g_z||_{L^\infty(\mathbb{C})}<k_3$. The function $g:\mathbb{C}\rightarrow\mathbb{C}$ satisfies $g(-z)=g(z)$, $g(\overline{z})=\overline{g(z)}$ for all $z\in\mathbb{C}$. The singular set of $g$ consists only of the critical values 

\[ \pm1 \emph{  and  } \left( \left(w_k+\delta_k\left(\frac{\delta_k}{m_k}\right)^{\left(\frac{1}{m_k-1}\right)}\left(\frac{m_k-1}{m_k}\right)\xi_j \right)_{j=1}^{m_k-1} \right)_{k=1}^{\infty} \]

\noindent and their copies under the symmetries $z\mapsto -z$, $z\mapsto\bar{z}$, where $(\xi_j)_{j=1}^{m_k-1}$ are the $(m_k-1)^\textrm{th}$ roots of $-1$.

\end{thm}

\begin{proof} The proof closely resembles the proof of Theorem 4.1 of \cite{2017arXiv171110629F}, but we summarize it as it is essential to the proof of Theorem \ref{main_theorem}. The bound $||g_{\overline{z}}/g_z||_{L^\infty(\cup D_n)}<k_3$ follows from considering $m_0''':=\textrm{max}\{m_0, m_0', m_0''\}$, $\delta_0$, $n_0$, $k_3:=\textrm{max}\{k_0, k_1, k_2\}$ for constants as in Lemma \ref{interpolation}, Proposition \ref{beta_extension}, Lemma \ref{rho} and Proposition \ref{small_support}. The extension of $g$ and the bound on $||g_{\overline{z}}/g_z||_{L^\infty(\mathbb{C})}$ are consequences of Theorem 7.2 of \cite{Bis15} (perhaps by increasing $k_3$) as described in Section 17 of \cite{Bis15} (see also Section 3 of \cite{MR3339086}). The symmetry $g(-z)=g(z)$, $g(\overline{z})=\overline{g(z)}$ is built into the definition of $g$. The singular values 

\[ \left(w_k+\delta_k\left(\frac{\delta_k}{m_k}\right)^{\left(\frac{1}{m_k-1}\right)}\left(\frac{m_k-1}{m_k}\right)\xi_j \right)_{j=1}^{m_k-1} \] 

\noindent arise from the critical values of $g|_{D_k}$. That the only other singular values of $f$ are reflected copies of the above critical values and $\pm1$ follows from Theorem 7.2 of \cite{Bis15}.
\end{proof}

\begin{rem}\label{localized_dilatation} Let $V:= \cup_{n=1}^{\infty} D_n$, and let $U:=\mathbb{C}\setminus \left( V \cup V^* \cup -V \cup -V^* \right)$, where $V^*$ denotes the complex conjugate of the set $V$.  Note that the extension of $g$ in $U$ is independent of a choice of $\boldsymbol{\delta}$, $\boldsymbol{R}$, $(\boldsymbol{r}_j)$, $\boldsymbol{w}$ since varying these parameters does not change the definition of $g$ on $\partial D_n$.
\end{rem}

The remainder of Section \ref{model_map} is dedicated to recording a number of results which roughly state that, for a class of parameters (which we will call \emph{permissible}), the behavior of $g$ and $f:=g\circ\phi^{-1}$ is sufficiently similar in some precise sense which we will need for the proof of Theorem \ref{main_theorem}.

\begin{definition}\label{delta_permissible} Let $\delta_0, n_0$ be as given in Theorem \ref{g_extension}. We call the parameters $\boldsymbol{\delta}$, $\boldsymbol{R}$, $(\boldsymbol{r}_j)$, $\boldsymbol{w}$ \emph{permissible} if $0\leq\boldsymbol{\delta}(k)<\delta_0$, $1\leq\boldsymbol{R}(k)<3/2$, $(\boldsymbol{r}_j(k))_{j=1}^{m_k-1}\in E_{\log\hspace{-2mm}\sqrt[\leftroot{-5}\uproot{2}n_0]{\boldsymbol{R}(k)}}$, and $\boldsymbol{w}(k)\in \overline{D}(0,3/4)$ for all $k\in\mathbb{N}$.
\end{definition}

\begin{prop}\label{normalization} There exist $\lambda_0\in\mathbb{R}$, $\boldsymbol{m}_0\in\mathbb{N}^{\mathbb{N}}$, $\boldsymbol{s}_0 \in (0,1)^{\mathbb{N}}$ such that if $\boldsymbol{\delta}$, $\boldsymbol{R}$, $(\boldsymbol{r}_j)$, $\boldsymbol{w}$ are permissible, $\lambda>\lambda_0$, $\boldsymbol{m}\geq\boldsymbol{m}_0$, and $(\emph{supp }g_{\overline{z}}) \cap D_n \subset \{ z \in D_n : |z-z_n| > \boldsymbol{s}_0(n)\}$ for all $n\in\mathbb{N}$, then there exist constants $a_1, a_0, a_{-1} \in \mathbb{C}$ such that

\begin{equation}\label{normalization_equation} \phi(z)=a_1z+a_0+\frac{a_{-1}}{z} + O\left(\frac{1}{|z|^2}\right) \emph{   as   } z\rightarrow\infty,
\end{equation}

\noindent where $\phi$ is any quasiconformal mapping as in Theorem \ref{MRT} such that $g\circ\phi^{-1}$ is holomorphic.

\end{prop}





\noindent The proof of Proposition 4.4 in \cite{2017arXiv171110629F} applies once one requires $\boldsymbol{s}_0(n)\rightarrow1^{-}$ sufficiently quickly as $n\rightarrow\infty$.

\begin{rem} Given $\phi$ as in Proposition \ref{normalization} satisfying (\ref{normalization_equation}), we may normalize $\phi$ so that:

\begin{equation}\label{hydrodynamical}  \phi(z)=z+\frac{a}{z} + O\left(\frac{1}{|z|^2}\right) \emph{   as   } z\rightarrow\infty
\end{equation}

\noindent for some $a\in\mathbb{C}$. This is the normalization we will always use henceforth. 

\end{rem}

\begin{prop}\label{close_to_id} For any $C>0$, $\varepsilon>0$, $T\geq1$, there exist $\lambda_0\in\mathbb{R}$, $\boldsymbol{m}_0\in\mathbb{N}^{\mathbb{N}}$, $\boldsymbol{s}_0 \in (0,1)^{\mathbb{N}}$, such that if $\lambda>\lambda_0$, $\boldsymbol{m}\geq\boldsymbol{m}_0$, the parameters $\boldsymbol{\delta}$, $\boldsymbol{R}$, $(\boldsymbol{r}_j)$, $\boldsymbol{w}$ are permissible, and $(\emph{supp }g_{\overline{z}}) \cap D_n \subset \{ z \in D_n : |z-z_n| > \boldsymbol{s}_0(n)\}$ for all $n\in\mathbb{N}$, then there exists a quasiconformal mapping $\phi: \mathbb{C} \rightarrow\mathbb{C}$ satisfying (\ref{hydrodynamical}) such that $g\circ\phi^{-1}$ is holomorphic and: 

\begin{equation}\label{qc_map_estimate} \left| \phi(z)-z \right| < \frac{C}{|z|} \emph{ for } |z|>T, \emph{ and }
\end{equation}
\begin{equation}\label{qc_map_estimate2} \left| \phi(z)-z \right| < \varepsilon \emph{ for all } z\in\mathbb{C}.
\end{equation}

\end{prop}

\noindent Again, the proof is the same normal family argument as the proof of Proposition 4.6 in \cite{2017arXiv171110629F}. The proofs of Proposition \ref{prop_on_derivative_of_f} and Corollary \ref{dist_to_1/2} below are also the same as the proofs of Proposition 4.11 and Corollary 4.14 (respectively) in \cite{2017arXiv171110629F}, and hence are omitted.

\begin{definition}\label{s_definition} Let $\lambda_0$, $\boldsymbol{m}_0$, $\boldsymbol{s}_0$ be as given in Proposition \ref{close_to_id} for $\varepsilon=\varepsilon_0:=1/32$, and $C=T=1$. We call the parameters $\lambda$, $\boldsymbol{m}$ \emph{permissible} if $\lambda>\lambda_0$ and $\boldsymbol{m}(k)>\boldsymbol{m_0}(k)$, for all $k\in\mathbb{N}$.
\end{definition}

\begin{rem}\label{Permissibility_remark} Permissible parameters $\lambda$, $\boldsymbol{\delta}$, $\boldsymbol{R}$, $(\boldsymbol{r}_j)$, $\boldsymbol{m}$, $\boldsymbol{w}$ determine a quasiregular function $g$ via (\ref{quasiregular_definition}) and Theorem \ref{g_extension}. If, in addition,  $(\textrm{supp }g_{\overline{z}}) \cap D_n \subset \{ z \in D_n : |z-z_n| > \boldsymbol{s}_0(n)\}$ for all $n\in\mathbb{N}$, then $\phi$ satisfies (\ref{qc_map_estimate}) and (\ref{qc_map_estimate2}) with $C=T=1$ and $\varepsilon=1/32$, where $\phi$ is a quasiconformal map normalized as in (\ref{hydrodynamical}) such that $g\circ\phi^{-1}$ is holomorphic.

\end{rem}

\begin{rem} We will henceforth begin considering the local inverse $g^{-1}$, which will always be defined in a neighborhood of $g(x)=y$ with $x,y>0$ such that $g^{-1}(y)=x$. There are no positive critical points of $g$ by (\ref{quasiregular_definition}) so that this inverse is always well defined, at least locally near $y$. The domain of the branch $g^{-1}$ will always map (under $g^{-1}$) to a subset of $S^+$. The same remarks apply to the local inverse $f^{-1}$.
\end{rem}

\begin{prop}\label{prop_on_derivative_of_f} Suppose that $g'(x)\geq2$ for $x\geq1/32$, that $\lambda$, $\boldsymbol{\delta}$, $\boldsymbol{R}$, $(\boldsymbol{r}_j)$, $\boldsymbol{m}$, $\boldsymbol{w}$ are permissible, and $\textrm{supp}(g_{\overline{z}}) \cap D_n \subset \{ z \in D_n : |z-z_n| > \boldsymbol{s}_0(n)\}$ for all $n\in\mathbb{N}$. Assume furthermore that $\textrm{supp}(g_{\overline{z}})\cap S^+\subset\{z\in S^+: \dist(z,\partial S^+) < 1/16\}$. Then 


\begin{align}\label{derivative_of_f} \left( \frac{(1/8)^2(1/4-2\varepsilon_0)}{(3/8)^2(1/4)} \right)^n \cdot \prod_{k=1}^n (g^{-1})'(g^k(1/2)+2\varepsilon_0) \leq \\ \left| (f^{-n})'(g^n(1/2)) \right|  \leq  \label{derivative_of_f1} \\
 \left( \frac{(5/8)^2(1/4+2\varepsilon_0)}{(3/8)^2(1/4)\lambda} \right)^n\cdot\frac{1}{\lambda-(\varepsilon_0/\lambda^{n-2})}. \label{derivative_of_f2}
\end{align}


\end{prop}

\begin{rem}\label{choice_of_lambda} We will henceforth fix $\lambda=\lambda_0$ as in Proposition \ref{close_to_id}, with several extra  conditions: we assume that $\lambda_0$ is sufficiently large so that $g'(x)\geq2$ for $x\geq 1/32$, and that $\lambda_0$ is sufficiently large so that $g^n(1/2)\rightarrow\infty$ as $n\rightarrow\infty$ (see Lemma 3.2 of \cite{MR3339086}). Furthermore, we assume $\lambda$ is sufficiently large so that $\textrm{supp}(g_{\overline{z}})\cap S^+\subset\{z\in S^+: \dist(z,\partial S^+) < 1/16\}$ (see the definition of $T(r_0)$ as in Theorem 1.1 of \cite{Bis15}). Lastly, we assume that $\lambda=\lambda_0$ is sufficiently large so that (\ref{derivative_of_f2}) tends to $0$ as $n\rightarrow\infty$. Note that (\ref{derivative_of_f}) and (\ref{derivative_of_f2}) are independent of permissible $\boldsymbol{\delta}$, $\boldsymbol{R}$, $(\boldsymbol{r}_j)$, $\boldsymbol{m}$, $\boldsymbol{w}$. Furthermore, observe that the map $g|_{S+}$ and the points $z_n$ as in (\ref{quasiregular_definition}) are now both fixed henceforth as they depend only on $\lambda$.

\end{rem}

\begin{definition}\label{sequence_along_which_we} Define the sequence $(\tilde{p}_n)_{n=1}^{\infty}$ such that $|z_{\tilde{p}_n}-g^n(1/2)|$ is minimized. \end{definition}

\begin{rem} The reason for the $\sim$ notation in Definition \ref{sequence_along_which_we} is that we will later more frequently use a subsequence of $(\tilde{p}_n)_{n=1}^\infty$ in Section \ref{fixpoint_section}: see Remark \ref{referee_suggestion}. It is for this subsequence of $(\tilde{p}_n)_{n=1}^\infty$ that we will later reserve the notation $(p_n)_{n=1}^\infty$. We also remark here that the sequence $(\tilde{p}_n)$ depends only on $\lambda$ (and in particular is independent of $\boldsymbol{\delta}$, $\boldsymbol{R}$, $(\boldsymbol{r}_j)$, $\boldsymbol{m}$, $\boldsymbol{w}$), and $\lambda$ has been fixed in Remark \ref{choice_of_lambda}.

\end{rem}

\begin{cor}\label{dist_to_1/2} There exists $n'\in\mathbb{N}$ such that if $\boldsymbol{\delta}$, $\boldsymbol{R}$, $(\boldsymbol{r}_j)$, $\boldsymbol{m}$, $\boldsymbol{w}$ are permissible, and $(\textrm{supp }g_{\overline{z}}) \cap D_k \subset \{ z \in D_k : |z-z_k| > \boldsymbol{s}_0(k)\}$ for all $k\in\mathbb{N}$, then $f^{-n}(D_{\tilde{p}_n})\subset \overline{D}(1/2, 1/8) \subset \overline{D}(0, 3/4)$ for all $n>n'$.
\end{cor}





\section{Quasiconformal Surgery and Fixpoints}
\label{fixpoint_section}

Recall that the function $g$ as defined in (\ref{quasiregular_definition}) and Theorem \ref{g_extension} depended on parameters $\lambda$, $\boldsymbol{\delta}$, $\boldsymbol{R}$, $(\boldsymbol{r}_j)$, $\boldsymbol{m}$, $\boldsymbol{w}$. Our goal is to assign values to the parameters $\lambda$, $\boldsymbol{\delta}$, $\boldsymbol{R}$, $(\boldsymbol{r}_j)$, $\boldsymbol{m}$, $\boldsymbol{w}$ such that the associated entire function $f:=g\circ\phi^{-1}$ is as in Theorem \ref{main_theorem}.   We have already fixed $\lambda$ in Remark \ref{choice_of_lambda}, and in Proposition \ref{selection_of_r_m} below we will assign values to the parameters $\boldsymbol{R}$, $\boldsymbol{m}$ by an inductive procedure. Later in this Section the parameters $\boldsymbol{\delta}$, $(\boldsymbol{r}_j)$, $\boldsymbol{w}$ will be assigned in Proposition \ref{penultimate} using the fixpoint Theorem \ref{fixpoint}. It is in Theorems \ref{fixpoint} and \ref{parameters}, used in this Section to control the orbits of singular values, where the approach in the present work differs most notably from \cite{2017arXiv171110629F}.

\begin{rem}\label{purely_expository} The following remark is purely expository and is meant to motivate the more technical aspects of Section \ref{fixpoint_section}.

 The difficulty of the proof of Theorem \ref{main_theorem} lies in choosing parameters for $g$ such that the critical values of $f:=g\circ\phi^{-1}$ escape to $\infty$, while still ensuring $f$ possesses the desired wandering domain. Figure  \ref{fig:general_strat} illustrates the desired behavior: we wish to choose parameters $\boldsymbol{\delta}$, $\boldsymbol{R}$, $(\boldsymbol{r}_j)$, $\boldsymbol{m}$, $\boldsymbol{w}$ so that, for instance, the critical points of $f|_{\phi(D_{1})}$ are eventually mapped (under $f$) to $+1$, whence $f^n(1)\rightarrow\infty$ as $n\rightarrow\infty$. It is not evident, however, how such a choice of parameters is to be achieved: the iterates of the critical points of $f|_{\phi(D_1)}$ \emph{depend} on a choice of the parameters $\boldsymbol{\delta}$, $\boldsymbol{R}$, $(\boldsymbol{r}_j)$, $\boldsymbol{m}$, $\boldsymbol{w}$. Moreover, this dependence is not explicit: the Beltrami coefficient of $g$, and hence the behavior of the correction map $\phi$, depend on a choice of $\boldsymbol{\delta}$, $\boldsymbol{R}$, $(\boldsymbol{r}_j)$, $\boldsymbol{m}$, $\boldsymbol{w}$. The straightening Theorem \ref{MRT} is not constructive, and so it is not evident what the pointwise dependence of iterates of $f$ is on the parameters $\boldsymbol{\delta}$, $\boldsymbol{R}$, $(\boldsymbol{r}_j)$, $\boldsymbol{m}$, $\boldsymbol{w}$. 
 
 The solution to this, as mentioned in the Introduction, is to employ the fixpoint Theorem \ref{fixpoint}: we will study the dependence of $f$ at a given point on the parameters $\boldsymbol{\delta}$, $\boldsymbol{R}$, $(\boldsymbol{r}_j)$, $\boldsymbol{m}$, $\boldsymbol{w}$. A fixpoint of an aptly chosen map (see (\ref{fixpoint_map}) and the proof of Proposition \ref{fixpoint_application}) will yield the desired choice of parameters. There are important hypotheses we need to verify in order to employ Theorem \ref{fixpoint}. To this end, we need to verify a number of estimates which we will discuss below in Remark \ref{referee_suggestion}. 
 
 \end{rem}
 
 \begin{figure}
\centering
\scalebox{.4}{
\begingroup%
  \makeatletter%
  \providecommand\color[2][]{%
    \errmessage{(Inkscape) Color is used for the text in Inkscape, but the package 'color.sty' is not loaded}%
    \renewcommand\color[2][]{}%
  }%
  \providecommand\transparent[1]{%
    \errmessage{(Inkscape) Transparency is used (non-zero) for the text in Inkscape, but the package 'transparent.sty' is not loaded}%
    \renewcommand\transparent[1]{}%
  }%
  \providecommand\rotatebox[2]{#2}%
  \newcommand*\fsize{\dimexpr\f@size pt\relax}%
  \newcommand*\lineheight[1]{\fontsize{\fsize}{#1\fsize}\selectfont}%
  \ifx\svgwidth\undefined%
    \setlength{\unitlength}{1098.06126392bp}%
    \ifx\svgscale\undefined%
      \relax%
    \else%
      \setlength{\unitlength}{\unitlength * \real{\svgscale}}%
    \fi%
  \else%
    \setlength{\unitlength}{\svgwidth}%
  \fi%
  \global\let\svgwidth\undefined%
  \global\let\svgscale\undefined%
  \makeatother%
  \begin{picture}(1,0.72106678)%
    \lineheight{1}%
    \setlength\tabcolsep{0pt}%
    \put(0,0){\includegraphics[width=\unitlength,page=1]{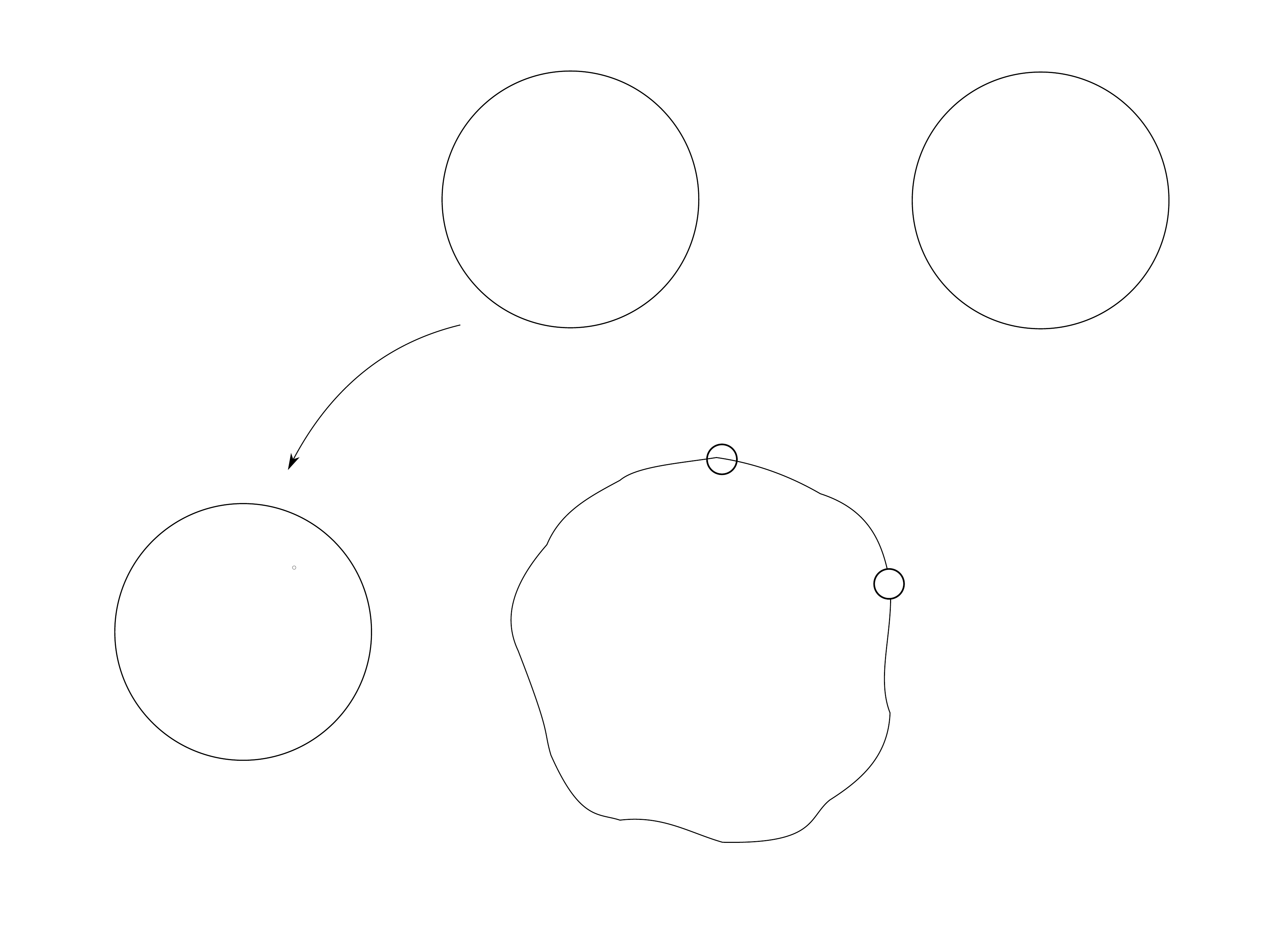}}%
    \put(0.09617136,0.32982987){\color[rgb]{0,0,0}\makebox(0,0)[lt]{\lineheight{0}\smash{\begin{tabular}[t]{l}\scalebox{3}{$\mathbb{D}$}\end{tabular}}}}%
    \put(0,0){\includegraphics[width=\unitlength,page=2]{general_strat_experiment.pdf}}%
    \put(0.31029491,0.65712436){\color[rgb]{0,0,0}\makebox(0,0)[lt]{\lineheight{0}\smash{\begin{tabular}[t]{l}\scalebox{3}{$D_{p_1}$}\end{tabular}}}}%
    \put(0.68488726,0.66376669){\color[rgb]{0,0,0}\makebox(0,0)[lt]{\lineheight{0}\smash{\begin{tabular}[t]{l}\scalebox{3}{$D_{p_{2}}$}\end{tabular}}}}%
    \put(0.23288167,0.42958768){\color[rgb]{0,0,0}\makebox(0,0)[lt]{\lineheight{0}\smash{\begin{tabular}[t]{l}\scalebox{3}{$g$}\end{tabular}}}}%
    \put(0.57313894,0.42058886){\color[rgb]{0,0,0}\makebox(0,0)[lt]{\lineheight{0}\smash{\begin{tabular}[t]{l}\scalebox{3}{$f^{-n_2}\circ\phi$}\end{tabular}}}}%
    \put(0,0){\includegraphics[width=\unitlength,page=3]{general_strat_experiment.pdf}}%
  \end{picture}%
\endgroup%
}
\caption{ This Figure depicts the behavior for $f$ as in Theorem \ref{main_theorem}. The critical points of $g|_{D_{p_1}}$ are marked with stars, and the white circles on $\partial D_{p_2}$ represent points which map to $+1$ under $g$. Thus, the figure depicts the desired situation in which critical values of $f|_{\phi(D_{p_1})}$ map after $n_2+1$ iterates of $f$ to $+1$, with sufficient contraction of $f|_{D_{p_1}}$ so that $f^{n_2+1}\circ\phi( (1/2)D_{p_1} ) \subset \phi((2/3)D_{p_2} )$.       }
\label{fig:general_strat}       
\end{figure}

\begin{rem}\label{referee_suggestion} In Proposition \ref{selection_of_r_m} below, we will fix a choice of $\boldsymbol{R}$, $\boldsymbol{m}$ so that for a large class of choices of $\boldsymbol{\delta}$, $(\boldsymbol{r}_j)$, $\boldsymbol{w}$, the resulting maps $g$ and $f:=g\circ\phi^{-1}$ satisfy certain estimates (\ref{eq0})-(\ref{eq4}) listed below. Before stating Proposition \ref{selection_of_r_m}, we will introduce each of these estimates with some brief motivation. Each of (\ref{eq0})-(\ref{eq4}) will be necessary either in verifying the hypotheses of the fixpoint Theorem \ref{fixpoint}, or in the construction of a wandering domain for $f$.



First, so that the estimates of the previous Section \ref{model_map} may be applied to the correction map $\phi$, we will need to establish that $g$ satisfies: \begin{align} \textrm{supp}(g_{\overline{z}})\cap D_{n} \subset \{ z\in D_n: |z-z_n| > \boldsymbol{s_0}(n) \} \emph{ for all }n\in\mathbb{N}. \phantom{bdfadsf=\,} \label{eq0} \end{align}


 
 



 In order to lighten notation, it will be useful to employ the following convention. Given a sequence $(n_k)_{k=1}^\infty$ of natural numbers, and the subsequence $(\tilde{p}_n)_{n=1}^\infty$ of Definition \ref{sequence_along_which_we}, we define the sequence $(p_k)_{k=1}^\infty$ by $p_k:=\tilde{p}_{n_k}$. We are omitting the dependence of $(p_k)_{k=1}^\infty$ on the sequence $(n_k)_{k=1}^\infty$ in our notation, but this will not cause confusion as we will only consider one sequence $(n_k)_{k=1}^\infty$ defined in Proposition \ref{selection_of_r_m}.
 
 In order for the map (\ref{fixpoint_map}) (see the proof of Proposition \ref{fixpoint_application}) to satisfy the hypotheses of the fixpoint Theorem \ref{fixpoint}, we will need to find a subsequence $(n_k)_{k=1}^\infty$ so that the following holds: \begin{align} \frac{f^{-n_k}\circ\phi(z_{p_k} + \xi) - f^{-n_k}\circ\phi(z_{p_k})}{\xi\cdot(f^{-n_k})'(g^{n_k}(1/2)) }  \in   \exp\left(\overline{D}\left(0, \log\hspace{-1.5mm}\sqrt[\leftroot{-1}\uproot{2}n_0]{\boldsymbol{R}(p_{k-1})}\right)\right) \textrm{ and }  \label{eq1}
\\ \frac{f^{-n_k}\circ\phi(z_{p_k}+\xi_{j+1})-f^{-n_k}\circ\phi(z_{p_k}+\xi_j)}{(f^{-n_k})'(g^{n_k}(1/2))(\xi_{j+1}-\xi_{j})} \in \exp\left(\overline{D}\left(0, \log\hspace{-1.5mm}\sqrt[\leftroot{-1}\uproot{2}n_0]{\boldsymbol{R}(p_{k-1})}\right)\right)\textrm{, }  \label{eq1'}
\\   \forall\xi \in \partial\mathbb{D} \textrm{, } k\geq2\textrm{, } 1\leq j \leq \boldsymbol{m}(p_{k-1})-1 \textrm{, where } (\xi_{j})_{j=1}^{\boldsymbol{m}(p_{k-1})-1}:=(-1)^{\frac{1}{\boldsymbol{m}(p_{k-1})-1}}. \nonumber
\end{align}

\noindent Next, in order to be able to apply Proposition \ref{small_support}, we will need to find positive constants $(C_k)_{k=2}^\infty$ so that the following estimate holds: \begin{align} C_{{k+1}} \leq \left(\boldsymbol{m}(p_k)\right)^{\frac{1}{\boldsymbol{m}(p_k)}}\left( \frac{\boldsymbol{m}(p_k)}{\boldsymbol{m}(p_k)-1} \left(f^{-n_{k+1}}\right)'\left(g^{n_{k+1}}(1/2)\right) \right)^{\frac{\boldsymbol{m}(p_k)-1}{\boldsymbol{m}(p_k)}} \leq \frac{1}{16} \emph{ for } k\geq1. \label{eq2} 
\end{align}
 
\noindent In order to be able to apply the estimate of Lemma \ref{rho}, we will need that: \begin{align} f^{-n_k}\circ\phi(z_{p_k}) \in \overline{D}(0,3/4)\textrm{  for  } k\geq2. \label{eq3}
\end{align}

\noindent Lastly, we will need the following estimate on the contraction of $f|_{D_{p_k}}$ in order to construct a wandering Fatou component: \begin{align}  
\left(\frac{k}{k+1}\right)^{\boldsymbol{m}(p_k)} + \frac{k}{k+1} \left(\boldsymbol{m}(p_k)\right)^{\frac{1}{\boldsymbol{m}(p_k)}}\left( \frac{\boldsymbol{m}(p_k)}{\boldsymbol{m}(p_k)-1} \left(f^{-n_{k+1}}\right)'\left(g^{n_{k+1}}(1/2)\right) \right)^{\frac{\boldsymbol{m}(p_k)-1}{\boldsymbol{m}(p_k)}} \phantom{} \label{eq4}
\\ < \inf_{\xi\in\partial\mathbb{D}} \left| f^{-n_{k+1}}\circ\phi\left(\frac{(k+1)\xi}{k+2}+z_{p_{k+1}}\right) - f^{-n_{k+1}}\circ\phi\left(z_{p_{k+1}}\right) \right|\emph{  for  } k\geq1. \phantom{af} \nonumber
\end{align}
 
\noindent The left-hand side of (\ref{eq4}) roughly signifies the contraction of $f|_{D_{p_k}}$, and the right-hand side of (\ref{eq4}) is roughly the radius of a disc contained in $f^{-n_{k+1}}(D_{p_{k+1}})$. Thus, establishing (\ref{eq4}) will allow us to conclude there is sufficient contraction of $f|_{D_{p_k}}$ in order to construct a wandering domain for $f$:  see the proof of Theorem \ref{main_theorem} at the end of this Section for details.



\end{rem}

\noindent With the above discussion, we now state the following:

\begin{prop}\label{selection_of_r_m} There exists a subsequence $(n_k)_{k=1}^{\infty}$ of natural numbers, a choice of permissible parameters $\boldsymbol{R}$, $\boldsymbol{m}$, and positive constants $(C_{k})_{k=2}^{\infty}$ such that: for any choice of permissible $(\boldsymbol{r}_j)$, $\boldsymbol{w}$, $\boldsymbol{\delta}$ with $1/16\geq\boldsymbol{\delta}(p_k)\geq C_{{k+1}}$ for all $k\in\mathbb{N}$ and $\boldsymbol{w}(l)=\boldsymbol{\delta}(l)=0$ for $l\in\mathbb{N}\setminus(p_k)_{k=1}^{\infty}$, the relations (\ref{eq0})-(\ref{eq4}) hold.


\end{prop}

\begin{proof} In order to define the sequences $(n_k)_{k=1}^\infty$, $(\boldsymbol{R}(p_k))_{k=1}^{\infty}$, $(\boldsymbol{m}(p_k))_{k=1}^{\infty}$, $(C_k)_{k=2}^{\infty}$, our logic will be as follows. We start by defining $n_1:=1$ so that $p_1:=\tilde{p}_{n_1}=\tilde{p}_1$. We first choose $\boldsymbol{R}(p_1)$, $n_2$, $C_{2}$, $\boldsymbol{m}(p_1)$ (in that order) so that (\ref{eq1}), (\ref{eq3}) hold with $k=2$ and (\ref{eq2}), (\ref{eq4}) hold with $k=1$ if $\boldsymbol{\delta}$, $\boldsymbol{R}$, $(\boldsymbol{r}_j)$, $\boldsymbol{m}$, $\boldsymbol{w}$ are any permissible extension of the choices $\boldsymbol{R}(p_1)$, $\boldsymbol{m}(p_1)$  and $1/16\geq\boldsymbol{\delta}(p_1)\geq C_{2}$, \emph{under the extra assumption} that (\ref{eq0}) holds. For each $l>1$, we then recursively choose $\boldsymbol{R}(p_l)$, $n_{l+1}$, $C_{{l+1}}$, $\boldsymbol{m}(p_l)$ (in that order) based on our previous choices $\boldsymbol{R}(p_k)$, $n_{k+1}$, $C_{{k+1}}$, $\boldsymbol{m}(p_k)$ for $1\leq k<l$, so that (\ref{eq1}), (\ref{eq3}) hold with $k=l+1$ and (\ref{eq1'}), (\ref{eq2}), (\ref{eq4}) hold with $k=l$ if $\boldsymbol{\delta}$, $\boldsymbol{R}$, $(\boldsymbol{r}_j)$, $\boldsymbol{m}$, $\boldsymbol{w}$ are any permissible extension of the choices $(\boldsymbol{R}(p_k))_{k=1}^l$, $(\boldsymbol{m}(p_k))_{k=1}^l$ and $1/16\geq\boldsymbol{\delta}(p_k)\geq C_{{k+1}}$ for $1\leq k \leq l$, \emph{under the extra assumption} that (\ref{eq0}) holds. This inductively defines the sequences $(n_k)_{k=1}^\infty$, $(\boldsymbol{R}(p_k))_{k=1}^{\infty}$, $(\boldsymbol{m}(p_k))_{k=1}^{\infty}$, $(C_k)_{k=2}^{\infty}$, whence we will be able to observe that this definition is such that (\ref{eq0}) indeed holds for any permissible $(\boldsymbol{r}_j)$, $\boldsymbol{w}$, $\boldsymbol{\delta}$ with $1/16\geq\boldsymbol{\delta}(p_k)\geq C_{{k+1}}$ for all $k\in\mathbb{N}$, and $\boldsymbol{m}(l)=\boldsymbol{m}_0(l)$,  $\boldsymbol{R}(l)=1$ and $\boldsymbol{w}(l)=\boldsymbol{\delta}(l)=0$ for $l\in\mathbb{N}\setminus(p_k)_{k=1}^{\infty}$.




As already mentioned, we define $n_1:=1$, so that $p_1=\tilde{p}_1$. (see Definition \ref{sequence_along_which_we}). Consider $\boldsymbol{s_0}(p_1)$ where $\boldsymbol{s_0}$ is as in Definition \ref{s_definition}. Define $\boldsymbol{R}(p_1)$ permissible so that $1<\boldsymbol{R}(p_1)<1/\boldsymbol{s_0}(p_1)$. Let $\xi\in\partial\mathbb{D}$ and let $\boldsymbol{\delta}$, $\boldsymbol{R}$, $(\boldsymbol{r}_j)$, $\boldsymbol{m}$, $\boldsymbol{w}$ be any permissible extension of $\boldsymbol{R}(p_1)$ such that (\ref{eq0}) holds. So as to ensure  (\ref{eq1}) for $k=2$, we consider

\begin{align}\label{distortion_quantity} \frac{f^{-n}(\phi(z_{\tilde{p}_n} + \xi)) - f^{-n}(\phi(z_{\tilde{p}_n}))}{(\xi)\cdot(f^{-n})'(g^{n}(1/2)) } = \left( \frac{f^{-n}(\phi(z_{\tilde{p}_n} + \xi)) - f^{-n}(\phi(z_{\tilde{p}_n}))}{(f^{-n})'(\phi(z_{\tilde{p}_n})) ( \phi(z_{\tilde{p}_n} + \xi) - \phi(z_{\tilde{p}_n}) )} \right) \cdot
\\  \cdot \left( \frac{\phi(z_{\tilde{p}_n} + \xi) - \phi(z_{\tilde{p}_n})}{\xi}  \right) \cdot \left( \frac{(f^{-n})'(\phi(z_{\tilde{p}_n}))}{(f^{-n})'(g^{n}(1/2))} \right). \nonumber 
\end{align}

\noindent The three terms on the right-hand side of (\ref{distortion_quantity}) tend to $1$ as $n\rightarrow\infty$ (apply Theorem \ref{thm:Koebe}(a) and Theorem \ref{thm:Koebe_argument}(a) to the first term,  apply (\ref{qc_map_estimate}) to the second term, and apply Theorem \ref{thm:Koebe}(b) and Theorem \ref{thm:Koebe_argument}(b) to the third term) uniformly over $\xi\in\partial\mathbb{D}$. This means we can find $n_2\in\mathbb{N}$ such that  (\ref{distortion_quantity}) with $n\geq n_2$ is contained in $\exp(\overline{D}(0,\log\hspace{-1.5mm}\sqrt[\leftroot{-1}\uproot{2}n_0]{R_{p_1}}))$. By Corollary \ref{dist_to_1/2}, we can further impose the condition that $n_2\in\mathbb{N}$ is chosen sufficiently large so that (\ref{eq3}) holds for $k=2$. In order to later prove (\ref{eq4}) when $k=1$, we need another condition on $n_2$, for which we consider the following expression: 

\begin{align}\label{wandering_estimate} \left| f^{-n}\left(\phi\left(\frac{2}{3}\xi+z_{\tilde{p}_{n}}\right)\right) - f^{-n}\big(\phi(z_{\tilde{p}_{n}})\big) \right| =  \left| \frac{f^{-n}\left(\phi\left(\frac{2}{3}\xi+z_{\tilde{p}_{n}}\right)\right) - f^{-n}\big(\phi(z_{\tilde{p}_{n}})\big)}{(f^{-n})'(\phi(z_{\tilde{p}_n}))\left(\phi\left(\frac{2}{3}\xi+z_{\tilde{p}_{n}}\right) - \phi(z_{\tilde{p}_{n}})\right)} \right| \cdot 
\\ \cdot \left| (f^{-n})'(g^n(1/2)) \right| \cdot \left|\frac{(f^{-n})'(\phi(z_{\tilde{p}_n}))}{(f^{-n})'(g^n(1/2))}\right| \cdot \left| \frac{ \phi\left(\frac{2}{3}\xi+z_{\tilde{p}_{n}}\right) - \phi(z_{\tilde{p}_{n}})}{\frac{2}{3}\xi+z_{\tilde{p}_{n}} - z_{\tilde{p}_{n}}} \right| \cdot \frac{2}{3}. \nonumber
\\ \nonumber
\end{align}

\noindent The first, third, and fourth terms of the right-hand side of (\ref{wandering_estimate}) tend to 1 as $n\rightarrow\infty$ (apply Theorem \ref{thm:Koebe}(a) to the first term, Theorem \ref{thm:Koebe}(b) to the third term, and (\ref{qc_map_estimate}) to the fourth term). Thus we can ensure that $n_2$ is sufficiently large so that 

\begin{align} \label{wandering_estimate2} \frac{2}{3}\left| \frac{f^{-n}\left(\phi\left(\frac{2}{3}\xi+z_{\tilde{p}_{n}}\right)\right) - f^{-n}\big(\phi(z_{\tilde{p}_{n}})\big)}{(f^{-n})'(\phi(z_{\tilde{p}_n}))\left(\phi\left(\frac{2}{3}\xi+z_{\tilde{p}_{n}}\right) - \phi(z_{\tilde{p}_{n}})\right)} \right| \cdot \left|\frac{(f^{-n})'(\phi(z_{\tilde{p}_n}))}{(f^{-n})'(g^n(1/2))}\right| \cdot
 \\ \cdot \left| \frac{ \phi\left(\frac{2}{3}\xi+z_{\tilde{p}_{n}}\right) - \phi(z_{\tilde{p}_{n}})}{\frac{2}{3}\xi+z_{\tilde{p}_{n}} - z_{\tilde{p}_{n}}} \right| > \frac{\frac{2}{3}+\frac{1}{2}}{2} \nonumber
\end{align}

\noindent with $n\geq n_2$. We need one last condition on $n_2$ for the purpose of later being able to prove (\ref{eq1'}) when $k=2$. Consider:


\begin{align}\label{convexity_estimate} \frac{f^{-n}(\phi(z_{\tilde{p}_{n}}+\xi_{j+1}))-f^{-n}(\phi(z_{\tilde{p}_{n}}+\xi_j))}{(f^{-n})'(g^{n}(1/2))(\xi_{j+1}-\xi_{j})} =  \\ \frac{f^{-n}(\phi(z_{\tilde{p}_{n}}+\xi_{j+1}))-f^{-n}(\phi(z_{\tilde{p}_{n}}+\xi_j))}{(f^{-n})'(\phi(z_{\tilde{p}_n}+\xi_j))\cdot( \phi(z_{\tilde{p}_{n}}+\xi_{j+1}) - \phi(z_{\tilde{p}_{n}}+\xi_j) )}\cdot
\nonumber  \frac{ \phi(z_{\tilde{p}_{n}}+\xi_{j+1}) - \phi(z_{\tilde{p}_{n}}+\xi_j)}{\xi_{j+1}-\xi_j} \cdot \\ \cdot \frac{(f^{-n})'(\phi(z_{\tilde{p}_n}+\xi_j))}{(f^{-n})'(g^{n}(1/2))}, \nonumber \phantom{asdfs}
\end{align}


\noindent where $(\xi_{j})_{j=1}^{m}$ are the $m^{\textrm{th}}$ roots of $-1$ for some $m\in\mathbb{N}$. Again, by Theorems \ref{thm:Koebe} and \ref{thm:Koebe_argument}, the first and third terms on the right-hand side of (\ref{convexity_estimate}) tend to $1$ as $n\rightarrow\infty$, independently of $m$. Thus we can ensure that $n_2$ is sufficiently large so that the product of the first and third terms in (\ref{convexity_estimate}) with $n\geq n_2$ is contained in $\exp(\overline{D}(0,(\log\hspace{-1.5mm}\sqrt[\leftroot{-1}\uproot{2}n_0]{R_{p_1}})/2))$. We would like to estimate the remaining term

\begin{equation}\label{troublesome_term}  \frac{ \phi(z_{\tilde{p}_{n}}+\xi_{j+1}) - \phi(z_{\tilde{p}_{n}}+\xi_j)}{\xi_{j+1}-\xi_j} \end{equation}

\noindent appearing on the right-hand side of (\ref{convexity_estimate}), so as to ensure (\ref{convexity_estimate}) is contained in 

\noindent $\exp(\overline{D}(0,\log\hspace{-1.5mm}\sqrt[\leftroot{-1}\uproot{2}n_0]{R_{p_1}}))$, but we will need to postpone this estimate until later in the proof, when $\boldsymbol{m}(p_{1})$ will already be fixed and we vary the parameter $\boldsymbol{m}(p_2)$. For now, we let $\eta>0$ be such that if 

\begin{align} \label{first_delta} \left|\arg \left( \frac{ \phi(z_{\tilde{p}_{n}}+\xi_{j+1}) - \phi(z_{\tilde{p}_{n}}+\xi_j)}{\xi_{j+1}-\xi_j} \right)\right| <\eta \textrm{, and } 
\\ \label{second_delta} 1-\eta< \left| \frac{ \phi(z_{\tilde{p}_{n}}+\xi_{j+1}) - \phi(z_{\tilde{p}_{n}}+\xi_j)}{\xi_{j+1}-\xi_j} \right| < 1+\eta \textrm{,}   \end{align} 

\noindent then (\ref{convexity_estimate}) is contained in $\exp(\overline{D}(0,\log\hspace{-1.5mm}\sqrt[\leftroot{-1}\uproot{2}n_0]{R_{p_1}}))$. Fix $\zeta\in\mathbb{C}\setminus\{0\}$ of sufficiently small modulus so that 

\begin{equation}\label{kobe_estimate_1} \hspace{0mm}\sqrt[\leftroot{-1}\uproot{2}1/4]{1-\eta} <  \frac{((1+\pi/2)/2 - |\zeta| )^2}{((1+\pi/2)/2)^2} \textrm{ and } \frac{((1+\pi/2)/2 + |\zeta| )^2}{((1+\pi/2)/2)^2} < \hspace{-1mm}\sqrt[\leftroot{-1}\uproot{2}1/4]{1+\eta}. 
\end{equation}

\noindent  Ensure, using (\ref{qc_map_estimate}), that $n_2$ is sufficiently large such that 

\begin{equation}\label{kobe_estimate_2} \hspace{-1mm}\sqrt[\leftroot{-1}\uproot{2}1/4]{1-\eta} <\left|\frac{\phi(\xi+\zeta)-\phi(\xi)}{(\xi+\zeta)-\xi}\right| < \hspace{-1mm}\sqrt[\leftroot{-1}\uproot{2}1/4]{1+\eta} \textrm{ for }  \xi \in \partial D_{\tilde{p}_{n_2}}. 
\end{equation}

\noindent The inequalities (\ref{kobe_estimate_1}) and (\ref{kobe_estimate_2}) will later be used in conjunction with Theorems \ref{thm:Koebe},  \ref{thm:Koebe_argument} to estimate (\ref{troublesome_term}). This concludes our definition of $n_2$, and hence the definition of $p_2:=\tilde{p}_{n_2}$.




Having fixed $n_2$, Proposition \ref{prop_on_derivative_of_f} gives a lower bound $C_{2} \leq |(f^{-n_2})'(g^{n_2}(1/2))|$ under the extra assumption of (\ref{eq0}), and this lower bound is independent of permissible $\boldsymbol{\delta}$, $\boldsymbol{R}$, $(\boldsymbol{r}_j)$, $\boldsymbol{m}$, $\boldsymbol{w}$. We now proceed to choose $\boldsymbol{m}(p_1)$. By Proposition \ref{small_support}, we can choose $\boldsymbol{m}(p_1)>\boldsymbol{m}_0(p_1)$ to be sufficiently large so that $\textrm{supp}(g_{\overline{z}})\cap D_{p_1} \subset \{ z\in D_{p_1}: |z-z_{p_1}| > \boldsymbol{s_0}(p_1) \}$ for $1/16\geq\boldsymbol{\delta}(p_1)\geq C_{2}$ and any permissible $(\boldsymbol{r}_j(p_1))_{j=1}^{m_{p_1}-1}$, $\boldsymbol{w}(p_1)$. Ensure furthermore that $m_{p_1}=\boldsymbol{m}(p_1)>2$ so that the lower bound in (\ref{eq2}) follows for $k=1$. The upper bound in (\ref{eq2}) is deduced from the upper bound in Proposition \ref{prop_on_derivative_of_f}. We impose another condition on our selection of $m_{p_1}$ for the purpose of proving (\ref{eq4}) for $k=1$. Note that the left-hand side of (\ref{eq4}) for $k=1$ tends to $(1/2)\left| (f^{-n_2})'(g^{n_2}(1/2)) \right|$ as $m_{p_1}\rightarrow\infty$ (the term $(f^{-n_2})'(g^{n_2}(1/2))$ depends on $m_{p_1}$, however the convergence $\left(1/2\right)^{m_{p_1}} + (m_{p_1})^{1/m_{p_1}}\left( (m_{p_1})x/(m_{p_1}-1) \right)^{(m_{p_1}-1)/m_{p_1}}(1/2) \rightarrow x/2$ as $m_{p_1}\rightarrow\infty$ is uniform over $x$ in the interval in which $(f^{-n_2})'(g^{n_2}(1/2))$ is contained by (\ref{eq2})). Thus by (\ref{wandering_estimate}) and (\ref{wandering_estimate2}), we may further ensure $m_{p_1}=\boldsymbol{m}(p_1)$ is chosen sufficiently large so that (\ref{eq4}) holds for $k=1$ and any permissible extension of  $\boldsymbol{R}(p_1)$, $\boldsymbol{m}(p_1)$ such that (\ref{eq0}) holds and $1/16\geq\boldsymbol{\delta}(p_1)\geq C_{2}$. Lastly, we ensure that $\boldsymbol{m}(p_1)$ is sufficiently large so that 

\begin{equation}\label{last_Kobe_estimate}  \hspace{-0mm}\sqrt{1-\eta} <  \frac{((1+\pi/2)/2)^2}{((1+\pi/2)/2+|\xi_{j+1}-\xi_j|)^2} \textrm{ and } \frac{((1+\pi/2)/2)^2}{((1+\pi/2)/2-|\xi_{j+1}-\xi_j|)^2} < \hspace{-1mm}\sqrt{1+\eta}, \end{equation}

\noindent where $(\xi_j)_{j=1}^{\boldsymbol{m}(p_1)-1}$ are the ordered $(\boldsymbol{m}(p_1)-1)^{\textrm{th}}$ roots of unity. This concludes the definition of $\boldsymbol{R}(p_1)$, $n_2$, $C_{2}$, $\boldsymbol{m}(p_1)$. 


We define $\boldsymbol{R}(p_2)$, $n_3$, $C_{3}$, $\boldsymbol{m}(p_2)$ similarly. Define $\boldsymbol{R}(p_2)$ permissible so that $1<\boldsymbol{R}(p_2)<1/\boldsymbol{s_0}(p_2)$. Again, since (\ref{distortion_quantity}) tends to $1$ as $n\rightarrow\infty$ for any permissible extension $\boldsymbol{\delta}$, $\boldsymbol{R}$, $(\boldsymbol{r}_j)$, $\boldsymbol{m}$, $\boldsymbol{w}$ such that (\ref{eq0}) holds, we can find $n_3\in\mathbb{N}$ with $n_3>n_2$ such that  (\ref{distortion_quantity}) with $n=n_3$ is contained in $\exp(\overline{D}(0,  \log\hspace{-1mm}\sqrt[\leftroot{-0}\uproot{2}n_0]{\boldsymbol{R}(p_2)})  )$. Proposition \ref{prop_on_derivative_of_f} gives a lower bound $C_{3} \leq |(f^{-n_3})'(g^{n_3}(1/2))|$. By Proposition \ref{small_support}, we can choose $\boldsymbol{m}(p_2)>\boldsymbol{m}_0(p_2)$ with $2(\boldsymbol{m}(p_1)-1) | \boldsymbol{m}(p_2)$ to be sufficiently large so that $\textrm{supp}(g_{\overline{z}})\cap D_{p_2} \subset \{ z\in D_{p_2}: |z-z_{p_2}| > \boldsymbol{s_0}(p_2) \}$ for $1/16\geq\boldsymbol{\delta}(p_2)\geq C_{3}$. Again, the lower bound in (\ref{eq2}) holds for $k=2$ since $m_{p_2}>2$, and the upper bound in (\ref{eq2}) follows from the upper bound in Proposition \ref{prop_on_derivative_of_f}. Ensuring $n_3$, $\boldsymbol{m}(p_2)$ are chosen so that (\ref{eq4}) also holds when $k=2$ is similar to the argument given when $k=1$. Lastly, we estimate (\ref{troublesome_term}) with $n=n_2$ and $(\xi_j)_{j=1}^{\boldsymbol{m}(p_1)-1}$ the ordered $(\boldsymbol{m}(p_1)-1)^{\textrm{th}}$ roots of unity. Since the dilatation of $\phi|_{\overline{D}(z_{p_2}, (1+\pi/2)/2)}$ vanishes as $\boldsymbol{m}(p_2)\rightarrow\infty$, by a normal family argument and Theorem \ref{Lehto-Virtanen} we may assume, for the purposes of estimating (\ref{troublesome_term}), that $\phi$ is conformal in ${\overline{D}(z_{p_2}, (1+\pi/2)/2)}$. Two applications of Theorem \ref{thm:Koebe}(a) together with the estimates (\ref{kobe_estimate_1}), (\ref{kobe_estimate_2}), and (\ref{last_Kobe_estimate}) then prove (\ref{second_delta}), and (\ref{first_delta}) is proven similarly. It follows that (\ref{eq1'}) holds for $k=2$. 


The rest of the sequences $(n_k)_{k=1}^\infty$, $(\boldsymbol{R}(p_k))_{k=1}^\infty$, $(\boldsymbol{m}(p_k))_{k=1}^\infty$, $(C_k)_{k=2}^{\infty}$ are chosen similarly. For any permissible extension $\boldsymbol{\delta}$, $\boldsymbol{R}$, $(\boldsymbol{r}_j)$, $\boldsymbol{m}$, $\boldsymbol{w}$ with $1/16\geq\boldsymbol{\delta}(p_k)\geq C_{{k+1}}$ for all $k\in\mathbb{N}$, under the extra assumption that (\ref{eq0}) holds, the relations (\ref{eq1}), (\ref{eq1'}), (\ref{eq2}), and (\ref{eq4}) follow from construction, and (\ref{eq3}) follows from Corollary \ref{dist_to_1/2}.

Now define $\boldsymbol{R}(l)=1$, $\boldsymbol{m}(l)=\boldsymbol{m}_0(l)$ for $l\in\mathbb{N}\setminus(p_k)_{k=1}^{\infty}$. This completes the definition of $\boldsymbol{R}$, $\boldsymbol{m}$. Note that if $\boldsymbol{\delta}(l)=\boldsymbol{w}(l)=0$ for $l\in\mathbb{N}\setminus(p_k)_{k=1}^{\infty}$, then $g|_{D_l}$ is holomorphic so that (\ref{eq0}) holds for such $l$. That (\ref{eq0}) holds for for any index $p_k$ and any permissible $(\boldsymbol{r}_j(p_k))_{j=1}^{m_{p_k}-1}$, $\boldsymbol{w}(p_k)$, $\boldsymbol{\delta}(p_k)$ with $1/16\geq\boldsymbol{\delta}(p_k)\geq C_{{k+1}}$ was ensured by the above selection of $\boldsymbol{R}(p_k)$, $\boldsymbol{m}(p_k)$. Thus, for our definitions of $(n_k)_{k=1}^{\infty}$ $\boldsymbol{R}$, $\boldsymbol{m}$, $(C_k)_{k=2}^{\infty}$, we have that if $(\boldsymbol{r}_j)$, $\boldsymbol{w}$, $\boldsymbol{\delta}$ are permissible with $1/16\geq \boldsymbol{\delta}(p_k)\geq C_{{k+1}}$ for all $k\in\mathbb{N}$ and $\boldsymbol{w}(l)=\boldsymbol{\delta}(l)=0$ for $l\in\mathbb{N}\setminus(p_k)_{k=1}^{\infty}$, then (\ref{eq0})-(\ref{eq4}) hold.



\end{proof}

\begin{rem} We henceforth fix $\boldsymbol{R}$, $\boldsymbol{m}$ as in the statement of Proposition \ref{selection_of_r_m}, and continue to use the sequences $(n_k)_{k=1}^\infty$, $(C_k)_{k=2}^\infty$ as given in Proposition \ref{selection_of_r_m}. Note that for $l\in\mathbb{N}\setminus(p_k)_{k=1}^\infty$, one has $\boldsymbol{m}(l)=\boldsymbol{m}_0(l)$, and $\boldsymbol{R}(l)=1$ by the proof of Proposition \ref{selection_of_r_m}. 
\end{rem}



\begin{prop}\label{fixpoint_application} Let $l\in\mathbb{N}$. There exist permissible $\boldsymbol{\delta}$, $(\boldsymbol{r}_j)$, $\boldsymbol{w}$ (depending on $l$) such that for $1\leq k\leq l$, one has:

\begin{equation}  f^{n_{k+1}}\left(w_{p_k}\right)=\phi\left(z_{p_{k+1}}\right)\textrm{, } 1/16\geq\boldsymbol{\delta}\left(p_k\right)\geq C_{{k+1}}\textrm{, and } f^{n_{k+1}+1}(g(c))=+1 \end{equation}

\noindent for any critical point $c$ of $g$ with $c \in D_{p_k}$.

\end{prop}

\begin{rem} Note that $g$ and $f$ share the same critical values ($f$ differs from $g$ by pre-composition with a homeomorphism), and that $f^n(1)\rightarrow\infty$ as $n\rightarrow\infty$.
\end{rem}

\begin{proof} Let us consider the case $l=k=1$. For $n\not= p_1$, define $\boldsymbol{w}(n)=\boldsymbol{\delta}(n)=0$ and $\boldsymbol{r}_j(n)=1$ for $1\leq j\leq m_n-1$. In order to choose $\boldsymbol{w}(p_1)$, $\boldsymbol{\delta}(p_1)$, $(\boldsymbol{r}_j(p_1))_{j=1}^{m_{p_1}-1}$, consider the following map:

\begin{align}\label{fixpoint_map} \overline{D}(0, 3/4) \times [C_{2}, 1/16] \times E_{\varepsilon_1} \longrightarrow \overline{D}(0, 3/4) \times [C_{2}, 1/16] \times E_{\varepsilon_1} \phantom{as}
\\ \begin{pmatrix}w_{p_1} \vspace{3mm}\\ \delta_{p_1}  \vspace{3mm} \\  \left(r_j\left(p_1\right)\right)_{j=1}^{m_{p_1}-1}  \end{pmatrix} \mapsto \nonumber \begin{pmatrix} f^{-n_2}\left(\phi\left(z_{p_2}\right)\right) \\   \left(m_{p_{1}}\right)^{\frac{1}{m_{p_{1}}}}\left( \frac{m_{p_{1}}}{m_{p_{1}}-1} \left(f^{-n_{2}}\right)'\left(g^{n_{2}}(1/2)\right) \right)^{\frac{m_{p_{1}}-1}{m_{p_{1}}}}\\ \left(\frac{f^{-n_2}\circ\phi(z_{p_2} + \xi_j) - f^{-n_2}\circ\phi(z_{p_2})}{\xi_j\cdot(f^{-n_2})'(g^{n_2}(1/2)) }  \right)_{j=1}^{m_{p_1}-1} \end{pmatrix} \nonumber
\end{align}

\noindent where we recall the notation $(\xi_j)_{j=1}^{m_{p_1}-1}$ for the $(m_{p_1}-1)^{\textrm{th}}$ roots of $-1$, and $\varepsilon_1:=\log(\hspace{-1mm}\sqrt[\leftroot{-1}\uproot{2}n_0]{\boldsymbol{R}(p_1)})$. Our goal is to find a fixpoint of (\ref{fixpoint_map}), since for such a fixpoint we would have

\begin{align}\label{cancellation} w_{p_1} + r_j(p_1)\cdot\xi_j\cdot\left(\delta_{p_1}\left(\frac{\delta_{p_1}}{{m_{p_1}}}\right)^{\left(\frac{1}{{m_{p_1}}-1}\right)}\left(\frac{{m_{p_1}}-1}{{m_{p_1}}}\right)\right) 
\\ = f^{-n_2}\left(\phi\left(z_{p_2}\right)\right) +  \left(\frac{f^{-n_2}(\phi(z_{p_2} + \xi_j)) - f^{-n_2}(\phi(z_{p_2}))}{\xi_j\cdot(f^{-n_2})'(g^{n_2}(1/2))}\right) \xi_j(f^{-n_2})'(g_{n_2}(1/2))  \nonumber
\\ = f^{-n_2}\left(\phi\left(z_{p_2} + \xi_j\right)\right), \nonumber
\end{align}

\noindent for any $1\leq j\leq m_{p_1}-1$. As $j$ ranges between $1$ and ${m_{p_1}}-1$, the left-hand side of (\ref{cancellation}) ranges over all critical values of $f$ arising from critical points of $g$ in $D_{p_1}$, whereas the right-hand side is mapped to $1$ by $f^{n_2+1}$ for each $1\leq j \leq {m_{p_1}}-1$, as 

\begin{equation} f^{n_2+1}(f^{-n_2}\left(\phi\left(z_{p_2} + \xi_j\right)\right))=g(z_{p_2} + \xi_j)=(\xi_j)^{\boldsymbol{m}(p_2)}=1, \end{equation} 

\noindent where the last equality holds since $2(m_{p_1}-1) | \boldsymbol{m}(p_2)$ as noted in the proof of Proposition \ref{selection_of_r_m} (see also Figure \ref{fig:general_strat}). 

Why does (\ref{fixpoint_map}) have a fixpoint? This is a consequence of Theorem \ref{fixpoint} once we have established the necessary hypotheses. Indeed, note that $E_{\varepsilon_1}$ is convex by Lemma \ref{convexity}, and so the domain of (\ref{fixpoint_map}) is convex because it is a product of convex sets. The image of $\left(w_{p_1}, \delta_{p_1}, (r_j(p_1))_{j=1}^{m_{p_1}-1}\right)$ under (\ref{fixpoint_map}) is contained in $\overline{D}(0,3/4) \times [C_{2}, 1/16] \times E_{\varepsilon_1}$ by Proposition \ref{selection_of_r_m}: (\ref{eq3}) ensures the first factor in the image is contained in $\overline{D}(0,3/4)$, (\ref{eq2}) ensures the second factor is contained in $[C_{2}, 1/16]$, and (\ref{eq1}), (\ref{eq1'}) ensure the third factor is contained in $E_{\varepsilon_1}$. Lastly, continuity of (\ref{fixpoint_map}) follows from Proposition \ref{continuity_into_L^infty} and Theorem \ref{parameters}. Namely, (\ref{fixpoint_map}) is a composition of two maps: the first is an $L^{\infty}(\mathbb{C})$-valued map sending any $\left(w_{p_1}, \delta_{p_1}, (r_j({p_1}))_{j=1}^{m_{p_1}-1}\right) \in \overline{D}(0, 3/4) \times [C_{2}, 1/16] \times E_{\varepsilon_1}$ to $g_{\overline{z}}/g_{z}$, and this is continuous by Proposition \ref{continuity_into_L^infty} and Remark \ref{localized_dilatation}. The second map in the composition maps $L^{\infty}(\mathbb{C})$ into $\overline{D}(0, 3/4) \times [C_{2}, 1/16] \times E_{\varepsilon_1}$, and is as described in Theorem \ref{parameters} (and in particular is continuous by Theorem \ref{parameters}). Thus the hypotheses of Theorem \ref{fixpoint} are satisfied, and so (\ref{fixpoint_map}) has a fixpoint, as needed.

For larger $l$ and $1\leq k\leq l$, a similar argument holds. Namely, for $t\in\mathbb{N}\setminus\{ p_1, p_2, ...., p_l \}$, one defines $\boldsymbol{w}(t)=\boldsymbol{\delta}(t)=0$ and $(\boldsymbol{r}_j(t))=1$ for $1\leq j\leq m_{t}-1$. A version of the mapping (\ref{fixpoint_map}) with $3l$ product factors is considered, and is continuous for completely analogous reasons to the case $l=1$, whence a fixpoint corresponds to a choice of

\begin{equation} \left(\boldsymbol{\delta}(p_k)\right)_{k=1}^l, \left(\left(\boldsymbol{r}_j\left(p_k\right)\right)_{j=1}^{m_{p_k}-1}\right)_{k=1}^l, (\boldsymbol{w}(p_k))_{k=1}^l \end{equation}

\noindent satisfying the conclusions of Proposition \ref{fixpoint_application}.

\end{proof}


\begin{prop}\label{penultimate} There exist permissible $\boldsymbol{\delta}$, $(\boldsymbol{r}_j)$, $\boldsymbol{w}$ such that for $1\leq k <\infty$, one has  

\begin{equation}\label{penultimate_relations}  f^{n_{k+1}}\left(w_{p_k}\right)=\phi\left(z_{p_{k+1}}\right)\textrm{, } 1/16\geq\boldsymbol{\delta}\left(p_k\right)\geq C_{{k+1}}\textrm{, and } f^{n_{k+1}+1}(g(c))=+1 \end{equation}

\noindent for any critical point $c$ of $g$ with $c \in D_{p_k}$. Furthermore, for $1\leq k <\infty$:

\begin{equation}\boldsymbol{\delta}(p_k) =  \left(\boldsymbol{m}(p_k)\right)^{\frac{1}{\boldsymbol{m}(p_k)}}\left( \frac{\boldsymbol{m}(p_k)}{\boldsymbol{m}(p_k)-1} \left(f^{-n_{k+1}}\right)'\left(g^{n_{k+1}}(1/2)\right) \right)^{\frac{\boldsymbol{m}(p_k)-1}{\boldsymbol{m}(p_k)}}.
\nonumber \end{equation}

\end{prop}

\begin{proof} For each $l<\infty$, Proposition \ref{fixpoint_application} guarantees the existence of parameters

\begin{align}\label{finite_solution} \left(\boldsymbol{w}^l(p_k), \boldsymbol{\delta}^l(p_k), \left(\boldsymbol{r}_j^l\left(p_k\right)\right)_{j=1}^{m_{p_k}-1}\right)_{k=1}^l  \in \prod_{k=1}^l \left( \overline{D}(0, 3/4) \times [C_{{k+1}}, 1/16] \times E_{\varepsilon_k} \right),
\\ \textrm{ with } \boldsymbol{w}(t)=\boldsymbol{\delta}(t)=0 \textrm{ and } (\boldsymbol{r}_j(t))_{j=1}^{m_{t}-1}=\boldsymbol{1}  \textrm{ for } t\in\mathbb{N}\setminus\{ p_1, p_2, ...., p_l \}\textrm{,}   \phantom{asd}     \nonumber
\end{align}

\noindent such that the associated entire function $f^l$ satisfies (\ref{penultimate_relations}) for $1\leq k\leq l$. We have emphasized notationally the dependence, for fixed $k$, of $ \left(\boldsymbol{w}^l(p_k), \boldsymbol{\delta}^l(p_k), (\boldsymbol{r}_j^l(p_k))_{j=1}^{m_{p_k}-1}\right)$ on $l$. The left-hand side of the first line in (\ref{finite_solution}) embeds into the compact space

\begin{equation}\label{infinite_product} \prod_{k=1}^\infty \left( \overline{D}(0, 3/4) \times [C_{{k+1}}, 1/16] \times E_{\varepsilon_k} \right),
\end{equation}

\noindent whence we can take a convergent subsequence. In other words, there exist $(\boldsymbol{\delta}(p_k))_{k=1}^\infty$, $\left((\boldsymbol{r}_j(p_k)\right)_{j=1}^{m_{p_k}-1})_{k=1}^\infty$, and $(\boldsymbol{w}(p_k))_{k=1}^\infty$ such that

\begin{equation}\label{q.c._convergence} \boldsymbol{w}^l(p_k) \xrightarrow{l\rightarrow\infty} \boldsymbol{w}(p_k)\emph{,  } \boldsymbol{\delta}^l(p_k) \xrightarrow{l\rightarrow\infty} \boldsymbol{\delta}(p_k)\emph{,  } (\boldsymbol{r}_j^l(p_k))_{j=1}^{m_{p_k}-1} \xrightarrow{l\rightarrow\infty} (\boldsymbol{r}_j(p_k))_{j=1}^{m_{p_k}-1}\emph{,  }
\end{equation}

\noindent for each $1 \leq k <\infty$, where we have suppressed the subsequence in $l$ to ease notation. We claim that the parameters 

\begin{align}\label{infinite_solution} (\boldsymbol{\delta}(p_k))_{k=1}^\infty\textrm{, } \left(\left(\boldsymbol{r}_j(p_k)\right)_{j=1}^{m_{p_k}-1}\right)_{k=1}^\infty\textrm{, } (\boldsymbol{w}(p_k))_{k=1}^\infty\textrm{, } \phantom{fasdf}
\\ \textrm{ with } \boldsymbol{w}(t)=\boldsymbol{\delta}(t)=0 \textrm{ and } (\boldsymbol{r}_j(t))_{j=1}^{m_{t}-1}=\boldsymbol{1}  \textrm{ for } t\in\mathbb{N}\setminus\left(p_k\right)_{k=1}^{\infty} \nonumber\end{align}

\noindent satisfy the conclusions of Proposition \ref{penultimate}. Let $\phi^l$, $\phi$ denote the quasiconformal mappings, normalized as in (\ref{hydrodynamical}), associated with the parameters as in (\ref{finite_solution}), (\ref{infinite_solution}), respectively. Note that $\phi^l_{\overline{z}}/\phi^l_{z} \rightarrow \phi_{\overline{z}}/\phi_{z}$ a.e. as $l\rightarrow\infty$ by (\ref{q.c._convergence}) and Remark \ref{localized_dilatation}, so that by taking a further subsequence in $l$ if necessary, we claim that

\begin{equation}\label{convergence_in_l} \phi^l \xrightarrow{l\rightarrow\infty} \phi \textrm{ uniformly on compact subsets of } \mathbb{C}.
\end{equation} 

\noindent Indeed, the maps $\phi^l$ converge in a subsequence to some quasiconformal map $\psi$ by normality of $(\phi^l)_{l=1}^\infty$, and since $\phi^l_{\overline{z}}/\phi^l_{z}$ converges a.e. to $\phi_{\overline{z}}/\phi_{z}$, one has $\psi_{\overline{z}}/\psi_z = \phi_{\overline{z}}/\phi_{z}$ a.e. by Theorem \ref{Lehto-Virtanen}, whence by the uniqueness of Theorem \ref{MRT} it follows that $\psi\equiv\phi$. 

Let $f:=g\circ\phi^{-1}$ be as associated with the parameters in (\ref{infinite_solution}). For any $k\in\mathbb{N}$ and $l>k$, (\ref{penultimate_relations}) holds true with $f$ replaced by $f^l:=g\circ(\phi^l)^{-1}$, whence the corresponding statements for $f$ follow from (\ref{q.c._convergence}) and (\ref{convergence_in_l}). The last conclusion in the statement of Proposition \ref{penultimate} also follows from (\ref{q.c._convergence}) and (\ref{convergence_in_l}).

\end{proof}


\begin{proof}[of Theorem \ref{main_theorem}] Take $f:=g\circ\phi^{-1}$ for parameters $\lambda$ as chosen in Remark \ref{choice_of_lambda}, $\boldsymbol{m}$, $\boldsymbol{R}$ as chosen in Proposition \ref{selection_of_r_m}, and  $\boldsymbol{w}$, $\boldsymbol{\delta}$, $(\boldsymbol{r}_j)$ as chosen in Proposition \ref{penultimate}. We claim that $\phi(D(z_{p_1}, 1/2))$ is contained in a wandering Fatou component for the map $f$. Note that: 

\begin{equation}\label{something} |g(z)-g(z_{p_1})| = |g(z)-f^{-n_2}\left(\phi\left(z_{p_2}\right)\right)| < \inf_{\xi\in\partial\mathbb{D}} \left| f^{-n_2}\left(\phi\left(\frac{2}{3}\xi+z_{p_2}\right)\right) - f^{-n_2}\left(\phi\left(z_{p_2}\right)\right) \right| \nonumber \end{equation}

\noindent for $z\in\{ z \in D_{p_1}: |z-z_{p_1}| < 1/2 \}$, since $g(z_{p_1})=f^{-n_2}(\phi(z_{p_2}))$ by Proposition \ref{penultimate}, and the inequality follows from (\ref{eq4}). In other words, $\phi^{-1}\circ f^{n_2+1} \circ \phi(D(z_{p_1}, 1/2)) \subset D(z_{p_2}, 2/3)$ (see Figure \ref{fig:general_strat}). Similar reasoning shows that 

\begin{equation} f^{n_{k+1}+1}\circ\phi\left(D\left(z_{p_k}, \frac{k}{k+1}\right)\right) \subset \phi\left(D\left(z_{p_{k+1}}, \frac{k+1}{k+2}\right)\right)\textrm{, for all } k \geq 1. \nonumber
\end{equation}

\noindent It follows then that the family of iterates $(f^n)_{n=1}^{\infty}$ is normal on $\phi(D(z_{p_k}, \frac{k}{k+1}))$ for any $k\geq1$, since any subsequence of $(f^n)_{n=1}^{\infty}$ is either bounded on $\phi(D(z_{p_k}, \frac{k}{k+1}))$, or contains a further subsequence converging to the constant limit function $\infty$. We claim that for $k\not = l$, $\phi(D(z_{p_k}, \frac{k}{k+1}))$ and $\phi(D(z_{p_l}, \frac{l}{l+1}))$ can not belong to the same Fatou component. To see this, note first that $\partial \phi(S^+)$ belongs to the Julia set of $f$, since $f(\phi(\partial S^+)) \subset [-1,1]$, and $[-1,1]$ iterates to $\infty$ under $f$. Next note that for $k\not=l$, there always exists some $m\in\mathbb{N}$ for which $f^m( \phi(D(z_{p_k}, \frac{k}{k+1})) ) \subset \phi(S^+)$ but $f^m( \phi(D(z_{p_l}, \frac{l}{l+1})) ) \not \subset \phi(S^+)$. Thus $f^m( \phi(D(z_{p_k}, \frac{k}{k+1})) )$ and $f^m( \phi(D(z_{p_l}, \frac{l}{l+1})) )$ are separated by the Julia set, and hence can not belong to the same Fatou component. Thus $\phi(D(z_{p_1}, 1/2))$ is contained in a wandering Fatou component for the map $f$. It remains to show that each singular value of $f$ escapes to infinity. Note that the critical values 

\[ \left( \left(w_{p_k}+\delta_{p_k}\left(\frac{\delta_{p_k}}{m_{p_k}}\right)^{\left(\frac{1}{m_{p_k}-1}\right)}\left(\frac{m_{p_k}-1}{m_{p_k}}\right)\xi_j \right)_{j=1}^{m_{p_k}-1} \right)_{k=1}^{\infty} \]

\noindent of $f$ iterate to $+1$ by Proposition \ref{penultimate}. Since $\boldsymbol{w}(t)=\boldsymbol{\delta}(t)=0$ for $t\in\mathbb{N}\setminus(p_k)_{k=1}^{\infty}$, the only other singular values of $f$ are $0$, $\pm1$ by Theorem \ref{g_extension}. Since $0$, $\pm1$ escape to $\infty$ under $f$, the Theorem is proven. 

\end{proof}

\end{document}